\documentclass[12pt]{article}
\usepackage[margin=1in]{geometry} 
\usepackage{amsmath,amsthm,amssymb}
\usepackage{comment}
\usepackage{graphicx}
\usepackage{subcaption}
\usepackage[font=small]{caption}

\usepackage{xcolor}
\usepackage{hyperref}
\hypersetup{
    colorlinks,
    linkcolor={blue!90!black},
    citecolor={blue!100!black},
    urlcolor={blue!80!black}
}

\usepackage{tikz}
\usepackage{float}
\usetikzlibrary{decorations.pathreplacing,calligraphy}

\usepackage{authblk}
\title{
Unveiling low-dimensional patterns induced by convex non-differentiable regularizers
}
\author[1]{Ivan Hejný}
\author[1]{Jonas Wallin}
\author[1,2]{Małgorzata Bogdan}
\author[1]{Michał Kos}
\affil[1]{Department of Statistics, Lund University}
\affil[2]{Institute of Mathematics, University of Wroclaw}
\date{}

\usepackage{titlesec}
\titleformat*{\section}{\bfseries}
\titleformat*{\subsection}{\small\bfseries}

\newcommand\independent{\protect\mathpalette{\protect\independenT}{\perp}}
\def\independenT#1#2{\mathrel{\rlap{$#1#2$}\mkern2mu{#1#2}}}

\theoremstyle{plain}
\newtheorem{theorem}{Theorem}[section]
\newtheorem{proposition}[theorem]{Proposition}
\newtheorem{lemma}[theorem]{Lemma}
\newtheorem{corollary}[theorem]{Corollary}

\theoremstyle{definition}
\newtheorem{definition}[theorem]{Definition}
\newtheorem{example}[theorem]{Example}
\newtheorem{remark}[theorem]{Remark}

\theoremstyle{definition}
\newtheorem*{assumptionA}{Assumption A} 

\usepackage[normalem]{ulem}
\begin{document}

\maketitle

\begin{abstract}

Popular regularizers with non-differentiable penalties, such as Lasso, Elastic Net, Generalized Lasso, or SLOPE, reduce the dimension of the parameter space by inducing sparsity or clustering in the estimators' coordinates. In this paper, we focus on linear regression and explore the asymptotic distributions of the resulting low-dimensional patterns when the number of regressors $p$ is fixed, the number of observations $n$ goes to infinity, and the penalty function increases at the rate of $\sqrt{n}$. While the asymptotic distribution of the rescaled estimation error can be derived by relatively standard arguments, 
convergence of patterns requires a separate proof, which is yet missing from the literature, even for the simplest case of Lasso. To fill this gap, 
we use the Hausdorff distance as a suitable mode of convergence for subdifferentials, resulting in the desired pattern convergence. Furthermore, we derive the exact limiting probability of recovering the true model pattern. This probability goes to 1 if and only if the penalty scaling constant diverges to infinity and the regularizer-specific asymptotic irrepresentability condition is satisfied. We then propose simple two-step procedures that asymptotically recover the model patterns, irrespective of whether the irrepresentability condition holds or not. 

Interestingly, our theory shows that Fused Lasso cannot reliably recover its own clustering pattern, even for independent regressors. It also demonstrates how this problem can be resolved by ``concavifying'' the Fused Lasso penalty coefficients. Additionally, sampling from the asymptotic error distribution facilitates comparisons between different regularizers. We provide short simulation studies showcasing an illustrative comparison between the asymptotic properties of Lasso, Fused Lasso, and SLOPE.

\end{abstract}
\section{Introduction}

Consider the linear model $y = X\beta^0 + \varepsilon$,
where $X \in  \mathbb{R}^{n\times p}$ is the design matrix, $\beta^0 \in \mathbb{R}^p$ is the vector of regression coefficients, and $\varepsilon \in \mathbb{R}^n$ is the random noise vector with independent identically distributed entries $\varepsilon_1,\ldots,\varepsilon_n$. We consider regularized estimators of the form
\begin{equation}\label{main objective}
\hat{\beta}_n = \underset{\beta\in\mathbb{R}^p}{\operatorname{argmin}}\; \dfrac{1}{2}\Vert y-X\beta\Vert_2^2 + f_n(\beta),
\end{equation}
where $f_n$ is a convex penalty function. Incorporating the penalty function often allows one to obtain unique solutions to the above minimization problem when $p>n$. However, the advantages of penalization are apparent even when $n>p$, where the penalty stabilizes the variance of the estimators and often substantially reduces their mean errors compared to the classical least squares estimators. Further reduction of the estimation and prediction error can be obtained in situations where the vector of coefficients belongs to some lower-dimensional space.
Identification of such lower-dimensional patterns is, in some cases, achieved by penalizing with non-differentiable penalty functions, such as those defined through some modification of the $\ell^1$ norm. These include the popular Lasso, Elastic Net, SLOPE, or Generalized Lasso, which comprises, in particular, the Fused Lasso. \cite{tibshirani1996regression, Tib2005fusedLasso, tibshirani2011solution, bogdan2015slope,
ZouHastie2005ElasticNet}.

These regularizers induce sparsity or clustering \footnote{ In the context of SLOPE, by a cluster of $\beta$, we mean a subset of $\{1, \dots,p\}$, where $\vert \beta_i\vert$ is constant. In the context of Fused Lasso, a cluster means a set of consecutive indices where $\beta_i$ have the same value.} in the estimate, enabling them to exploit  the ``underlying structure'' of the signal vector $\beta^0$, which can improve the estimation properties of $\hat{\beta}$. 
For example, Lasso has the ability to recover zero elements of $\beta^0$, by correctly estimating them, or at least some of them, as $0$. Fused Lasso can additionally discover consecutive clusters in $\beta^0$, by setting consecutive values in the estimate to the same value. SLOPE has the ability to recover the most refined patterns in $\beta^0$, including clusters, where signs might differ, and the coordinates of the cluster are non-consecutive. For example, in
\begin{align*}
   \beta^0 = [\textcolor{purple}{1.7}, \textcolor{purple}{1.7}, 2.3, \textcolor{purple}{1.7}, \textcolor{blue}{0}],
\end{align*}
the purple cluster in $\beta^0$ is discoverable\footnote{By this we mean that if the covariates are not strongly correlated, SLOPE has positive probability that $\hat{\beta}_1=\hat{\beta_2}=\hat{\beta}_4$. This probability is zero for Fused Lasso/Lasso. } by SLOPE, but not by Fused Lasso, which can only cluster the first two coefficients. In this example, Lasso can reduce the dimension by $1$, Fused Lasso by $2$, and SLOPE by $3$. 

In the aforementioned examples, the regularizers share a common form 
\begin{equation}\label{general linear form intro}
    f(\beta)=\max\{v_1^T\beta,\dots,v_k^T\beta \} +g(\beta),
\end{equation}
where $v_1,\dots,v_k$ are the regularizer specific vectors in $\mathbb{R}^p$, and $g(\beta)$ is a convex differentiable function. The structural information that the regularizer $f$ can access about $\beta$ is captured by 
the \textit{pattern} of $f$ at $\beta$, defined as the set of indices, that maximizes $f(\beta)$
\begin{equation*}
    I_f(\beta):=\operatorname{argmax}_{i\in\{1,\dots,k\}}v_i^T\beta + g(\beta).
\end{equation*}

The above definition of pattern corresponds to the notion of an ``active index set'' in \cite{Lewis2002ActiveSets}. When $ g = 0 $, $ f(\beta) $ is a polyhedral gauge and the pattern can be equivalently defined using the subdifferential \cite{schneider2022geometry, graczyk2023pattern}. We define the \textit{set of all patterns} of $f$ as the (finite) image of the pattern map $I_f:\mathbb{R}^p\rightarrow \mathcal{P}(\{1,\dots,k\})$\footnote{$\mathcal{P}(\{1,\dots,k\})$ is the power set of $\{1,\dots,k\}$.}, and denote it by $\mathfrak{P}_f=\{I_f(\beta):\beta\in\mathbb{R}^p\}$. If the regularizer $f$ is known from the context, we drop the subscripts and write $I=I_f$ and $\mathfrak{P}=\mathfrak{P}_f$. 


\begin{figure}[ht]
    \centering
\begin{tikzpicture}[scale=1.4]

\draw [-latex] (1,0)-- (7,0) node[right = 5pt, black]{$\textcolor{black}{\boldsymbol{\gamma}}$};

\draw [-latex, dash pattern = on 2pt off 1pt, color = blue ](3,0.8)--(3,0.1);
\draw [-latex, dash pattern = on 2pt off 1pt ](6,0.8)--(6,0.1);

\begin{scriptsize}

\draw [fill = black] (1,0) circle (1pt) node[above = 5pt, black]{$0$};
\draw (3,0) circle (1pt);
\draw (6,0) circle (1pt);

\draw [decorate, decoration = {calligraphic brace, mirror, amplitude=5pt}] (1,-0.2) --  (3-0.05,-0.2);

\draw [decoration ={calligraphic brace, mirror, amplitude=5pt}, decorate] (3+0.05,-0.2) -- (6,-0.2);

\node[align=left] at (1.9,-1) {$\gamma<1/2$\\[4pt] $\hat{\beta}_n \sim OLS$};

\node[align=left] at (3.3,1.2) { \textcolor{blue}{$\hat{\beta}_n$ clusters} \\[4pt]$\textcolor{blue}{\boldsymbol{\gamma}=1/2}$};

\node[align=left] at (4.5,-0.9) {\textcolor{purple}{$1/2<\boldsymbol{\gamma}<1$}};

\node[align=left] at (6.5,1.2) { $\hat{\beta}_n$ inconsistent \\[4pt]$\gamma=1$};

\end{scriptsize}

\end{tikzpicture}
\caption{Scaling regimes, when $\hat{\beta}_n$ minimizes $\Vert y-X\beta\Vert/2+n^{\gamma}f(\beta)$, for fixed $p$ and $n\rightarrow\infty$. For $\gamma<1/2$, $\hat{\beta}_n$ is asymptotically equivalent to the OLS, there is no model selection. For $\gamma=1/2$, $\sqrt{n}(\hat{\beta}_n-\beta^0)$ converges in distribution, and $lim_{n\rightarrow\infty}\mathbb{P}[I(\hat{\beta}_n)=I(\beta^0)]\in(0,1)$. When $1/2<\gamma<1$, $\sqrt{n}(\hat{\beta}_n-\beta^0)$ diverges, and $\mathbb{P}[I(\hat{\beta}_n)=I(\beta^0)]$ converges to $1$, if the irrepresentability condition holds. }
\label{figure scaling regimes}
\end{figure}
We shall say that $\hat{\beta}_n$ recovers the $f-$ pattern of $\beta^0$, if $I_f(\hat{\beta}_n)=I_f(\beta^0)$.  Great effort has been made in the last two decades to establish conditions under which the model pattern is recovered. This turns out to be a nontrivial issue, even in the setup with $p$ fixed and $n$ going to infinity. To study this asymptotics, we consider scaling the regularizer in (\ref{main objective}) as $f_n(\beta)=n^{\gamma}f(\beta)$, where $f$ is some fixed penalty function. In terms of model selection, there are essentially two interesting scaling regimes, corresponding to $\gamma=1/2$ and $\gamma\in(1/2,1)$, summarized in Figure \ref{figure scaling regimes}.
Under strong scaling, when $f_n$ grows as $n^{\gamma}$ for $1/2<\gamma<1$, conditions for the exact recovery of support for the Lasso estimator have been explored in \cite{zhao2006model}, \cite{jia2015preconditioning} and for the recovery of the pattern with SLOPE in \cite{bogdan2022pattern}. A general approach to model consistency of partly smooth regularizers was developed in \cite{vaiter2017model}. In this paper, we investigate the ``classical'' weak scaling regime, when the penalty scaling is exactly of order $\sqrt{n}$. When $f_n \sim n^{1/2} f$, the error $\hat{u}_n=\sqrt{n}(\hat{\beta}_n-\beta^0)$ converges to a limiting distribution as in \cite{fu2000asymptotics}, and the probability that (\ref{main objective}) asymptotically recovers its own true model is strictly between $0$ and $1$. A limiting distribution of error $\hat{u}_n$, which selects patterns with positive probability, exists only when $\gamma=1/2$. This also enables some quantitative comparison of various methods based on their estimation and model selection accuracy. 

One of the primary contributions of the paper lies in providing a full characterization of the asymptotic model selection properties of a wide range of regularized estimators, when the penalty $f_n$ increases at the rate $n^{1/2}$ and $p$ remains fixed. To the best of our knowledge, a formal account on the model selection properties in this regime is still missing. In the seminal paper by Knight and Fu \cite{fu2000asymptotics}, the authors show that the rescaled error $\hat{u}_n=\sqrt{n}(\hat{\beta}_n-\beta^0)$ converges weakly to a limiting distribution $\hat{u}$. 
However, this does not imply weak convergence of $sgn(\hat{u}_n)$ to $sgn(\hat{u})$, because the sign function is discontinuous, rendering the continuous mapping theorem inapplicable. To elucidate the subtlety of the matter, we demonstrate that $sgn(\hat{u}_n)$ fails to converge weakly to $sgn(\hat{u})$ when $\hat{u}_n$ is the error obtained by penalizing with the convex penalty given by $\max\{\beta_1,\beta_2^2\}$, (see Appendix~\ref{appendix counterexample pattern convergence}). We are not aware of any rigorous argument in the literature that addresses this issue, not even in the simplest case of Lasso. This motivates the following definition:
\begin{definition}
    Let $f$ be a regularizer of the form (\ref{general linear form intro}) and $(\hat{u}_n)_{n\in\mathbb{N}}$ a sequence of random vectors in $\mathbb{R}^p$. We say that $\hat{u}_n$ converges \textit{weakly in $f-$pattern} to a random vector $\hat{u}$ if 
    \begin{equation}\label{pattern convergence}
        \lim_{n\rightarrow \infty}\mathbb{P}[I_f(\hat{u}_n)=\mathfrak{p}]=\mathbb{P}[I_f(\hat{u})=\mathfrak{p}],
    \end{equation}
    for every pattern $\mathfrak{p}\in\mathfrak{P}_f$.
\end{definition}

  One of our main contributions is showing that for a regularizer $f$ of the form~(\ref{general linear form intro}), the errors $\hat{u}_n=\sqrt{n}(\hat{\beta}_n-\beta^0)$ converge weakly in pattern to the asymptotic error $\hat{u}$, see Theorem~\ref{sqrt-asymptotic}, Theorem~\ref{main theorem pattern convergence of convex sets}, Corollary~\ref{asymptotic sampling corollary general}.

We derive the probability of recovering the true pattern in the low-dimensional limit for regularizers of the form~(\ref{general linear form intro}) (Theorem~\ref{general pattern recovery}). In relation to this, we investigate the irrepresentability condition \cite{zhao2006model, vaiter2017model, bogdan2022pattern}, under which correct recovery occurs with probability converging to 1 as penalty scaling increases (Corollary \ref{pattern recovery exponential bound}). 

Famously, under suitable penalty scaling, Adaptive Lasso \cite{HuiZouAdaptiveLasso} recovers the true model with probability going to one, irrespective of covariance $C$. A different second-order method, designed to recover the pattern, has been proposed in \cite{graczyk2023pattern} based on thresholding an initial estimate. We expand on this idea and show that the proposed two-step procedure (\ref{two-step procedure}) recovers the true model pattern with high probability, regardless of the covariance structure of the regressors (Lemma~ \ref{two-step proximal method lemma}, Theorem ~\ref{two-step proximal method theorem}). The procedure uses an initial estimate of $\beta^0$ and then regularizes it with penalty $f$. Finally, we apply the general theory to show that, under the independence of the regressors, Fused Lasso cannot recover all its patterns, even for strong penalty scaling. As a remedy, we suggest the Concavified Fused Lasso (Proposition~\ref{concavification of Fused Lasso}), by ``concavifying'' the tuning parameters of the Fused Lasso, which surprisingly yields exact pattern recovery of the signal. The auxiliary proofs and results are given in Appendix~\ref{Appendix}.

\section{Asymptotic distribution for the standard loss }

 Consider the linear model $y = X\beta^0 + \varepsilon$, with $X=(X_1,..,X_n)^{T}$, where $X_1, X_2, \dots$ are i.i.d. centered random vectors in $\mathbb{R}^p$ with the covariance matrix $C$. Further assume $\varepsilon_1, \varepsilon_2, \dots$ are i.i.d. centered random variables with variance $\sigma^2$ and $X\independent\varepsilon$. We begin by considering the minimizer (\ref{main objective}) for an arbitrary sequence of convex functions $f_n$:

 \begin{align}\label{SLOPE-estimator_n}
     \hat{\beta}_n &= \underset{\beta\in\mathbb{R}^p}{\operatorname{argmin }}  \dfrac{1}{2}\Vert y-X\beta\Vert^2_{2} + f_n(\beta)\\
     &=\underset{\beta\in\mathbb{R}^p}{\operatorname{argmin }}\dfrac{1}{2}(\beta-\beta^0)^{T}X^{T}X(\beta-\beta^0)-(\beta-\beta^0)^{T}X^{T}\varepsilon+\Vert\varepsilon\Vert_{2}^2/2 + f_n(\beta).\label{expanded slope estimator}
 \end{align}
 
 For fixed $p$ and $n\rightarrow\infty$, it follows from the law of large numbers and central limit theorem that:
 \begin{equation}\label{W_n C_n convergence}
     C_n:=\dfrac{1}{n}X^{T}X\overset{a.s.}{\longrightarrow} C\hspace{0.4cm}\text{ and }\hspace{0.4cm} W_n:=\dfrac{1}{\sqrt{n}}X^{T}\varepsilon\overset{d}{\longrightarrow} W\sim\mathcal{N}(0,\sigma^2C),
 \end{equation}
 and $\Vert
 \varepsilon\Vert^2_2/n\overset{a.s.}{\longrightarrow}\sigma^2$. Furthermore, we recall the definition of the directional derivative of a function $f:\mathbb{R}^p\rightarrow \mathbb{R}$ at a point $x$ in direction $u$:
 \begin{equation*}
     f'(x; u):= \operatorname{lim}_{t\downarrow 0} \dfrac{f(x+t u)-f(x)}{t}.
 \end{equation*}
 For convex $f$, the directional derivative always exists (see, for example, Theorem 23.1 \cite{rockafellar2015convex} ).
 The following statement and proof are direct extensions of the results in \cite{fu2000asymptotics}.

\begin{theorem}\label{sqrt-asymptotic}
Let $f:\mathbb{R}^p\rightarrow \mathbb{R}$ be any convex penalty function and $f_n=n^{1/2}f$. Assume $C$ is positive definite. Then $\hat{u}_n:= \sqrt{n}(\hat{\beta}_n-\beta^0)\overset{d}{\longrightarrow}\hat{u}$, where
\begin{align}
    \hat{u}&:=\textup{argmin}_{u} V(u),\nonumber\\
    V(u) &= \dfrac{1}{2}u^{T}Cu-u^{T}W+ f'({\beta^0};u)\label{V(u)},
\end{align}
with $W\sim\mathcal{N}(0,\sigma^2C)$, and $f'({\beta^0};u)$ the directional derivative of $f$ at $\beta^0$ in direction $u$. 

More generally, the result holds for any sequence of convex penalties of the form $f_n = n^{1/2} (f + \rho_n)$, such that $\rho_n(\beta)\rightarrow 0$ for every $\beta$, and $\rho_n$ are Lipschitz continuous with Lipschitz constants $c_n\rightarrow 0$ as $n\rightarrow\infty$. 
\end{theorem}

\begin{proof}
Substituting $u=\sqrt{n}(\beta-\beta^0)$ in (\ref{expanded slope estimator}), it follows that $\hat{u}_n$ minimizes the convex objective function
\begin{equation}\label{V_n(u)}
     V_n(u):=\dfrac{1}{2n}u^{T}X^{T}Xu-u^{T}\dfrac{1}{\sqrt{n}} X^{T}\varepsilon + f_n(\beta^0+u/\sqrt{n})-f_n(\beta^0).
\end{equation}
 Let us now study the asymptotic behavior of $V_n(u)$. The first two terms in $V_n(u)$ converge by (\ref{W_n C_n convergence}) and Slutsky to $u^{T}Cu/2-u^{T}W$. 

The Lipschitz constants $c_n$ of $\rho_n$ converge to zero, which yields, for sufficiently large $n$,
\begin{equation*}
n^{1/2}\vert\rho_n(\beta^0+u/\sqrt{n})-\rho_n(\beta^0)\vert \leq c_n \Vert u\Vert \rightarrow 0.
\end{equation*}

Thus, 
\begin{align*}
    f_n(\beta^0+u/\sqrt{n})-f_n(\beta^0) & = n^{1/2}(f(\beta^0+u/\sqrt{n})-f(\beta^0)) + n^{1/2}(\rho_n(\beta^0+u/\sqrt{n})-\rho_n(\beta^0))
\end{align*}
converges to the directional derivative $f'(\beta^0;u)$.
This, and another application of Slutsky, shows that $V_n(u)\overset{d}{\longrightarrow}V(u)$ for every $u\in\mathbb{R}^n$. By the convexity and uniqueness of the minimizer of $V(u)$, we obtain weak convergence of the minimizers $\hat{u}_n\overset{d}{\longrightarrow}{\hat{u}}$ by replicating the argument in Theorem 2 \cite{fu2000asymptotics}. 
\end{proof}
We are aware that the last step in the proof of Theorem 2 \cite{fu2000asymptotics} refers to an unpublished manuscript. However, the conclusion also follows by combining Theorem 3.2.2 in \cite{vanderVaart1996} with Problem 1.6.1 \cite{vanderVaart1996}.

The minor generalization from $f_n=n^{1/2}f$ to $f_n=n^{1/2}(f+\rho_n)$ in Theorem \ref{sqrt-asymptotic} covers penalty sequences of the form  
$f_n(\beta)=\lambda^n\Vert \beta\Vert$, where $\Vert \cdot \Vert$ is any norm on $\mathbb{R}^p$ and $\lambda^n/\sqrt{n}\rightarrow \lambda\geq 0$.
The asymptotics for the Lasso and Ridge regularizers are covered in \cite{fu2000asymptotics}. For the Ridge penalty $f(\beta)=\lambda\Vert \beta\Vert_2^2/2$, a direct calculation yields $f'(\beta^0;u)=\lambda\sum_{i=1}^p\beta^0_i u_i=\lambda u^T\beta^0$. The asymptotic error $\hat{u}$ minimizes $u^TCu/2-u^T W + \lambda u^T\beta^0$, and hence $\hat{u}=C^{-1}(W-\lambda\beta^0)\sim\mathcal{N}(-\lambda C^{-1}\beta^0,\sigma^2C^{-1})$.

Furthermore, the objective (\ref{V_n(u)}) with $f_n=n^{1/2}f$ and the objective (\ref{V(u)}) give optimality conditions for $\hat{u}_n$ and  $\hat{u}$ respectively: 
\begin{align}\label{main optimality condition}
    0&\in \partial_u V_n(u) = C_nu-W_n + \partial f(\beta^0+u/\sqrt{n}),\nonumber\\ 
    0&\in \partial_u V(u)\hspace{0,1cm} = Cu-W +\partial_u f'({\beta^0};u),
\end{align}
where $C_n, W_n$ are as in (\ref{W_n C_n convergence}) and we denote by $\partial_u f'(\beta^0;u)$ the subdifferential of the function $u\mapsto f(\beta^0; u)$.
Note that for a convex function $g:\mathbb{R}^p\rightarrow \mathbb{R}$ the subdifferential at $u\in\mathbb{R}^p$ is the set
\begin{equation*}
    \partial g(u)=\{v\in\mathbb{R}^p: g(u) + \langle v, \tilde{u}-u\rangle \leq g(\tilde{u})\hspace{0,1cm}\forall \tilde{u}\in\mathbb{R}^p\}. 
\end{equation*}
Describing the subdifferential for shrinkage estimators can be non-trivial and will be studied in the next section on subdifferential and pattern. 
However, if the proximal operator
\begin{equation*}
\text{prox}_{f'(\beta^0; \cdot)}(y):=\underset{u\in\mathbb{R}^p}{\operatorname{argmin }}\hspace{0.2cm} (1/2)\Vert u-y\Vert_2^2 + f'(\beta^0; u),
\end{equation*}
of the directional derivative $u\mapsto f'(\beta^0; u)$ is known, we can use proximal methods to solve the optimization problem (\ref{V(u)}). The proximal operator of the directional SLOPE derivative is described in the Appendix \ref{proximal operator}, and used for simulations in Section \ref{simulations}. We also refer the reader to \cite{pmlr-v206-larsson23a}, where the directional derivative of the SLOPE penalty is used for a coordinate descent algorithm for SLOPE.

\section{Pattern and Subdifferential}
So far, we have established weak convergence of the error $\hat{u}_n$. However, this does not guarantee any type of control over the clustering/sparsity behavior of the regularizer. We refer the reader to Appendix \ref{appendix counterexample pattern convergence} for an example where $\hat{u}_n$ converges to $\hat{u}$ in distribution, but $sgn(\hat{u}_n)$ fails to converge in distribution to $sgn(\hat{u})$. In fact, clusters, or more generally, model patterns, can be broken by infinitesimal perturbations that are ``invisible'' to the convergence in distribution. This necessitates a new approach, which relies on studying the subdifferential of the regularizer. It turns out that the limiting behavior of model patterns will be determined by (\ref{main optimality condition}), as long as $\partial_uf_n(\beta^0+u/\sqrt{n})$ converges in the Hausdorff distance to $\partial f'(\beta^0;u)$. We note that a related mode of set convergence, the Painlevé--Kuratowski convergence, was used by Geyer in \cite{geyer1994asymptotics} to study the asymptotics of constrained M-estimators and later revisited in \cite{shapiro2000asymptotics}. The notion of a general pattern was already introduced and explored in \cite{graczyk2023pattern}. We use a slightly different but equivalent definition of patterns through ``active'' sets \cite{Lewis2002ActiveSets}.  

Given a finite index set $\mathcal{S}$, we consider a penalty of the form
\begin{align}\label{general linear framework}
     f(x) &= \max\{\langle v_i, x\rangle : i\in\mathcal{S}\} + g(x),
\end{align}
where $g$ is continuously differentiable, convex and $v_i$ are finitely many distinct vertices in $\mathbb{R}^p$ such that $\forall i\in\mathcal{S}, \;v_i\notin con\{v_j: j \neq i\}$. We define the \textit{pattern} of the penalty $f$ at $x$ as the set of indices 
\begin{align*}
I(x)&=\operatorname{argmax}_{i\in\mathcal{S}}\langle v_i,x \rangle.
\end{align*}
We shall denote the set of all patterns $\mathfrak{P}=\{I(x): x\in \mathbb{R}^p\}$, and its elements interchangeably by $\mathfrak{p}, \mathfrak{p}_x$ or $I(x)$. Importantly, under (\ref{general linear framework}), one can verify that the sets of constant pattern $I^{-1}(\mathfrak{p})=\{x\in\mathbb{R}^p: I(x)=\mathfrak{p}\}$ are convex. 
Note that the pattern $I(x)$ only depends on the polyhedral gauge $h(x)=\max\{\langle v_i, x\rangle : i\in\mathcal{S}\}$. Moreover, its subdifferential is given by $\partial h(x) = con\{v_i:i\in I(x)\}$ and 
\begin{equation}\label{pattern as constant subdifferential}
    I(x)=I(y)\iff \partial h(x) =\partial h(y).
\end{equation} The equivalence in (\ref{pattern as constant subdifferential}) fully characterizes the patterns as the sets of constant subdifferential,  which is used as a definition of patterns in \cite{graczyk2023pattern}. By (\ref{pattern as constant subdifferential}), there is also one-to-one correspondence between patterns $I(x)$ and lower-dimensional faces $\partial h(x)$ of the polytope $\partial h(0)$, see \cite{schneider2022geometry}. 
Our framework allows for additional penalization with a smooth regularizer $g$. 
The subdifferential of  (\ref{general linear framework}) is
\begin{align}\label{subdifferential as convex hull}
\partial f(x) &= con\{v_i:i\in I(x)\}+\nabla g(x).
\end{align}
Further, one can show (see for instance \cite{hiriart2013convex, rockafellar2009variational}) that the directional derivative of $f$ at $x$ in direction $u$ satisfies
\begin{align}\label{directional derivative equation}
    f'(x;u) &= 
    \operatorname{max}_{i\in I(x)}\langle v_i,u\rangle + \langle \nabla g(x), u \rangle,\\
    \partial_u f'(x;u)&=con\{ v_i:i\in I_x(u)\}+\nabla g(x),\nonumber
\end{align}
where
\begin{align*}
    I_x(u)&:=\operatorname{argmax}_{i\in I(x)} \langle  v_i,u\rangle. 
\end{align*}
We call $I_x(u)$ the \textit{limiting pattern} of $u$ with respect to $x$. This is motivated by the fact that $I_x(u)=\lim_{\varepsilon\downarrow0}I(x+\varepsilon u)$, see Appendix \ref{appendix limiting pattern}.


We remark that if $f:\mathbb{R}^m\rightarrow\mathbb{R}$ is of the form (\ref{general linear framework}), then for any linear map $\psi:\mathbb{R}^p\rightarrow\mathbb{R}^m$, the composition $f\circ \psi:\mathbb{R}^p\rightarrow\mathbb{R}$ is also of the form (\ref{general linear framework}) and
\begin{equation}\label{composition of patterns}
    I_{f\circ \psi}(x)=I_{f}(\psi(x)),\hspace{0,8cm} \partial (f\circ \psi)(x)=\psi^T\partial f(\psi(x)).
\end{equation}
Also, any $f$ satisfying (\ref{general linear framework}) is partly smooth relative to the set of constant pattern $\mathcal{M}=I^{-1}(\mathfrak{p})$, for any pattern $\mathfrak{p}\in\mathfrak{P}$, see \cite{vaiter2017model, Lewis2002ActiveSets}.  

\subsection{Hausdorff distance}\label{Hausdorff distance and directional SLOPE derivative}

This section discusses some facts about Hausdorff distance, used later in proofs. It can be skipped if the reader wants to go directly to the main results in further sections.

Let $d(x,y)=\Vert x-y \Vert_2$ denote the standard Euclidean distance on $\mathbb{R}^p$. For $B\subset\mathbb{R}^p$, denote $B^{\delta}:=\{x\in\mathbb{R}^d: d(x,B)\leq\delta\}=\overline{B}+\overline{B_{\delta}(0)}$, where $\overline{B_{\delta}(x)}$ is the closed $\delta$-ball around $x$. For non-empty sets $A,B\subset\mathbb{R}^p$, the Hausdorff distance, also called the Pompeiu-Hausdorff distance, is defined as
\begin{equation*}
    d_H(A,B):=inf\{\delta\geq 0\vert A\subset B^{\delta}, B\subset A^{\delta}\};
\end{equation*}
see \cite{rockafellar2009variational}. The Hausdorff distance defines a pseudo-metric and yields a metric on the space of all closed, non-empty subsets of some bounded set $X\subset\mathbb{R}^p$. For a sequence of sets $A_n$, we write $A_n \overset{d_H}{\longrightarrow} A$, if $d_H(A_n,A)\rightarrow 0$ as $n\rightarrow\infty$. Convergence in the Hausdorff metric coincides with the Painlevé--Kuratowski convergence when the sequence $A_n$ is contained in a bounded set $X$. The Hausdorff metric is suitable for dealing with subdifferentials of real-valued convex  functions since these are compact. 
Note that for convergent sequences $A_n \overset{d_H}{\longrightarrow} A, A_n'\overset{d_H}{\longrightarrow} A'$, we have $A_n+A_n'\overset{d_H}{\longrightarrow} A+A'$, thus for any $\delta>0,\;$ $A_n^{\delta}\overset{d_H}{\longrightarrow}A^{\delta}$. 
Finally, for finitely many convergent sequences; $x_n^i\rightarrow x^i,\;i\in\mathcal{S}$, of points in $\mathbb{R}^p$:
\begin{equation}\label{corner convergence}
    con\{x^i_n: i\in\mathcal{S}\}\overset{d_H}{\longrightarrow}con\{x^i: i\in\mathcal{S}\},
\end{equation}
where $\vert\mathcal{S}\vert<\infty$. In particular, convergence of convex sets in Hausdorff distance is compatible with convergence in distribution of random vectors. The proof of the following Lemma is given in Appendix \ref{apppendix proofs}.

\begin{lemma}\label{Hausdorff lemma}
    Suppose $B_n\overset{d_H}{\longrightarrow}B$, where $B_n$ and $B$ are convex sets in $\mathbb{R}^p$. If $W_n\overset{d}{\longrightarrow}W$ for some $W$ with a continuous bounded density w.r.t. the Lebesgue measure on $\mathbb{R}^p$, then $\mathbb{P}[W_n\in B_n]\longrightarrow \mathbb{P}[W\in B]$.
\end{lemma}

\begin{lemma}\label{Hausdorff subdifferential lemma}
    Let $f$ be as in (\ref{general linear framework}). Then for $f_n=n^{1/2}f$;
\begin{align}\label{general Hausdorff convergence}
  \partial_u f_n(x+u/\sqrt{n})\overset{d_H}{\longrightarrow}\partial_u f'(x;u). 
\end{align}

\end{lemma}
\begin{proof}
By chain rule, (\ref{subdifferential as convex hull}), and (\ref{stabilization of limiting pattern}) we have
\begin{align*}
     \partial_u f_n(x+u/\sqrt{n}) & = \partial f(x+u/\sqrt{n})\\
     &=con \left\{ v_i: i\in I(x+u/\sqrt{n})\right\} + \nabla g(x+u/\sqrt{n})\\
     & \overset{d_H}{\longrightarrow} con \left\{ v_i: i\in I_x(u)\right\} + \nabla g(x) = \partial_u f'(x;u)
\end{align*}
where the last equality is (\ref{directional derivative equation}). 
\end{proof}

We remark that (\ref{general Hausdorff convergence}) holds more generally for 
$f_n(x)=n^{1/2}(\max\{\langle v_i^n, x\rangle: i\in\mathcal{S}\}+g(x))$, where $v_i^n\rightarrow v_i$ for each $1\leq i\leq N$, provided that for some $M>0$; $I_n(x)=I(x)$ for every $x\in\mathbb{R}^p$ and $n\geq M$. The proof of this follows by the same argument as Lemma~\ref{Hausdorff subdifferential lemma} and (\ref{corner convergence}). This covers, for example, the case of SLOPE $f_n(x)=J_{\boldsymbol{\lambda}^n}(x) = \max\{ \langle P\boldsymbol{\lambda}^n,x\rangle: P\in\mathcal{S}_p^{+/-}\}$, where $\boldsymbol{\lambda}^n/\sqrt{n}\rightarrow \boldsymbol{\lambda}$ with strictly decreasing non-negative $\boldsymbol{\lambda}$. The condition $I_n(x)=I(x)$ for every $x\in\mathbb{R}^p$ is satisfied, as long as $\boldsymbol{\lambda}^n$ is a strictly decreasing vector. For details on the SLOPE pattern and subdifferential, see Appendix \ref{SLOPE subdifferential Appendix}. 

\subsection{Weak pattern convergence}
The following theorem strengthens the weak convergence of the error $\hat{u}_n=\sqrt{n}(\hat{\beta}_n-\beta^0)$, established in Theorem \ref{sqrt-asymptotic}. From now on, we assume that the penalty $f$ satisfies (\ref{general linear framework}), and that $f_n=n^{1/2}f$ in (\ref{main objective}).

\begin{theorem}\label{main theorem pattern convergence of convex sets}
For every convex set $\mathcal{K}\subset\mathbb{R}^p$: $\mathbb{P}[\hat{u}_n\in\mathcal{K}]\longrightarrow\mathbb{P}[\hat{u}\in\mathcal{K}]$ as $n\rightarrow\infty$. In particular, $\hat{u}_n$ converges weakly in pattern to $\hat{u}$:
\begin{equation*}
    \mathbb{P}[I(\hat{u}_n)=\mathfrak{p}]\xrightarrow[n\rightarrow\infty]{} \mathbb{P}[I(\hat{u})=\mathfrak{p}],
\end{equation*}
for any pattern \footnote{ The pattern function $I$ does not have to be induced by the same penalty $f$, which defines the minimizer $\hat{u}_n$, but by any penalty satisfying (\ref{general linear framework}). } $\mathfrak{p}\in\mathfrak{P}$.
\end{theorem}
\begin{proof}
For any $\varepsilon>0$, by tightness of $\hat{u}_n$ there exists $M>0$ s.t. $\mathbb{P}[\hat{u}_n\in\mathcal{K}\setminus B_{M}(0)]<\varepsilon $ $\forall n\in\mathbb{N}$, where $B_M(0)=\{u\in\mathbb{R}^p:\Vert u\Vert\leq M\}$ is the ball of radius $M$.
Consider the finite partition of $\mathbb{R}^p$ into convex sets of constant pattern $I^{-1}(\mathfrak{p})=\{u\in\mathbb{R}^p: I(u)=\mathfrak{p}\}$, for $\mathfrak{p}\in\mathfrak{P}$. For each $\mathfrak{p}\in\mathfrak{P}$ consider the bounded convex set $\mathcal{K}^{\mathfrak{p}}= \mathcal{K}\cap I^{-1}(\mathfrak{p})\cap B_{M}(0)$. Note that for large $n$, the subdifferential $\partial f(\beta^0+u/\sqrt{n})$ and $\partial_u f'(\beta^0;u)$ do not depend on the choice of $u\in\mathcal{K}^{\mathfrak{p}}$. Fixing any vector $u_{\mathfrak{p}}\in\mathcal{K}^{\mathfrak{p}}$, we get by optimality conditions $(\ref{main optimality condition})$:
\begin{align*}
    \mathbb{P}\left[\hat{u}_n\in \mathcal{K}^{\mathfrak{p}}\right] = & \mathbb{P}\left[W_n\in C_n \mathcal{K}^{\mathfrak{p}}+\partial f\left(\beta^0 +u_{\mathfrak{p}}/\sqrt{n}\right)\right]\\
    \rightarrow 
    & \mathbb{P}\left[W\in C\mathcal{K}^{\mathfrak{p}}+\partial_u f'(\beta^0;u_{\mathfrak{p}})\right]=\mathbb{P}\left[\hat{u}\in \mathcal{K}^{\mathfrak{p}}\right],
\end{align*}
where $C_n$ and $W_n$ are given by (\ref{W_n C_n convergence}) and the convergence follows by Lemmas \ref{Hausdorff lemma} and \ref{Hausdorff subdifferential lemma}. Consequently,
\begin{align*}
    \limsup_{n\rightarrow\infty}\mathbb{P}\left[\hat{u}_n\in \mathcal{K}\right]
    &\leq \limsup_{n\rightarrow\infty} \sum_{\mathfrak{p}\in\mathfrak{P}}\mathbb{P}\left[\hat{u}_n\in \mathcal{K}^{\mathfrak{p}}\right]+\varepsilon
    = \sum_{\mathfrak{p}\in\mathfrak{P}}\mathbb{P}\left[\hat{u}\in \mathcal{K}^{\mathfrak{p}}\right]+\varepsilon
    \leq\mathbb{P}\left[\hat{u}\in \mathcal{K}\right]+\varepsilon,
    \\
    \liminf_{n\rightarrow\infty}\mathbb{P}\left[\hat{u}_n\in \mathcal{K}\right]
    &\geq \liminf_{n\rightarrow\infty} \sum_{\mathfrak{p}\in\mathfrak{P}}\mathbb{P}\left[\hat{u}_n\in \mathcal{K}^{\mathfrak{p}}\right]
    = \sum_{\mathfrak{p}\in\mathfrak{P}}\mathbb{P}\left[\hat{u}\in \mathcal{K}^{\mathfrak{p}}\right]
    \geq\mathbb{P}\left[\hat{u}\in \mathcal{K}\right]-\varepsilon,
\end{align*}
which shows that $\hat{u}_n$ converges to $\hat{u}$ on all convex sets. Finally, setting $\mathcal{K}=I^{-1}(\mathfrak{p})$ for some $\mathfrak{p}\in\mathfrak{P}$, gives the weak convergence of $I(\hat{u}_n)$ to $I(\hat{u})$.
\end{proof}
As a consequence of Theorem \ref{main theorem pattern convergence of convex sets}, we can characterize the asymptotic distribution of $I(\hat{\beta}_n)$ in terms of $\hat{u}$. Recall that $\hat{\beta}_n = \beta^0 + \hat{u}_n/\sqrt{n}$.

\begin{corollary}\label{asymptotic sampling corollary general}
For any pattern $\mathfrak{p}\in\mathfrak{P}$,
\begin{equation*}
    \mathbb{P}[I(\hat{\beta}_n)=\mathfrak{p}]\xrightarrow[n\rightarrow\infty]{}\mathbb{P}[I_{\beta^0}(\hat{u})=\mathfrak{p}],
\end{equation*}
where $I_{\beta^0}(\hat{u})=lim_{\varepsilon\downarrow 0}I(\beta^0+\varepsilon \hat{u})$.
\end{corollary}
\begin{proof}
Since $I(\hat{u}_n)\overset{d}{\rightarrow}I(\hat{u})$, by the Skorokhod representation theorem there exists a sequence of random patterns $I_n'\overset{d}{=}I(\hat{u}_n)$ and $I'\overset{d}{=}I(\hat{u})$ on a common probability space such that $I'_n\rightarrow I'$ almost surely. This means that eventually $I'_n = I'$ almost surely for all $n\geq N_1$ for some $N_1\in\mathbb{N}$, because the set of patterns is finite. Now, for each pattern $\mathfrak{p}\in\mathfrak{P}$ we can pick a representative vector of that pattern, $u_{\mathfrak{p}}\in I^{-1}(\mathfrak{p})$, and define $\tilde{u}_n:=u_{\mathfrak{p}}\iff I_n'=\mathfrak{p}$, and $\tilde{u}:=u_{\mathfrak{p}}\iff I'=\mathfrak{p}$. Consequently, $I(\tilde{u}_n)=I_n'\overset{d}{=}I(\hat{u}_n)$, $I(\tilde{u})=I'\overset{d}{=}I(\hat{u})$, and $\tilde{u}_n=\tilde{u}$ almost surely for all $n\geq N_1$. 

Let $\varepsilon>0$ be fixed, then by the tightness of $\hat{u}_n$ there is $M>0$ s.t. $\mathbb{P}[\Vert\hat{u}_n\Vert>M]<\varepsilon$ $\forall n$. Also, there is $N_2\in\mathbb{N}$ s.t. for all $n\geq N_2$, and $\Vert u\Vert\leq M$; $I(\beta^0+u/\sqrt{n})=I_{\beta^0}(u)$, by the definition of the limiting pattern $I_{\beta^0}(u)$. Therefore, if $\Vert\hat{u}_n\Vert\leq M$ and $n\geq \max\{N_1, N_2\}$, then
\begin{equation*}
I(\hat{\beta}_n)=I(\beta^0+\hat{u}_n/\sqrt{n})=I_{\beta^0}(\hat{u}_n) \overset{d}{=} I_{\beta^0}(\tilde{u}_n) \overset{a.s.}{=} I_{\beta^0}(\tilde{u})\overset{d}{=}I_{\beta^0}(\hat{u}).
\end{equation*}
Therefore, for any $\mathfrak{p}\in\mathfrak{P}$;
\begin{align*}
\limsup_{n\rightarrow\infty}\mathbb{P}[I(\hat{\beta}_n)=\mathfrak{p}]&\leq\limsup_{n\rightarrow\infty}\mathbb{P}[I(\hat{\beta}_n)=\mathfrak{p}, \Vert\hat{u}_n\Vert\leq M]+\varepsilon\leq \mathbb{P}[I_{\beta^0}(\hat{u})=\mathfrak{p}]+\varepsilon\\
\liminf_{n\rightarrow\infty}\mathbb{P}[I(\hat{\beta}_n)=\mathfrak{p}]&\geq\liminf_{n\rightarrow\infty}\mathbb{P}[I(\hat{\beta}_n)=\mathfrak{p}, \Vert\hat{u}_n\Vert\leq M]\geq \mathbb{P}[I_{\beta^0}(\hat{u})=\mathfrak{p}]-\varepsilon,
\end{align*}
which proves the claim.
\end{proof}

\subsection{Pattern recovery}
Conditions for pattern recovery for partly smooth regularizers were explored in \cite{vaiter2017model}. For SLOPE, an exact formula for the probability of asymptotic pattern recovery was established in Theorem 4.2 i)\cite{bogdan2022pattern}. Here, we harness the asymptotic formula (\ref{V(u)}) to provide a general proof for the asymptotic pattern recovery.

For a set $M\subset\mathbb{R}^p$, denote the parallel space $par(M)=span\{u-v: u,v\in M\}$. Assume that $f$ is a penalty satisfying (\ref{general linear framework}) and let $x\in\mathbb{R}^p,\mathfrak{p}_x\in\mathfrak{P}$ such that $I(x)=\mathfrak{p}_x$.  We define the \textit{pattern space of $f$ at $x$},  as the vector space
\begin{align*}
    \langle U_x\rangle:= span\{I^{-1}(\mathfrak{p}_x)\}.
\end{align*}

The following are equivalent representations of the pattern space $\langle U_x\rangle$:
    \begin{align}
        i) &\hspace{0.3cm}span\{I^{-1}(\mathfrak{p}_x)\},\nonumber\\
        ii) &\hspace{0.3cm}par(\partial f(x))^{\perp},\nonumber\\
        iii) &\hspace{0.3cm}\{u\in\mathbb{R}^p: I_x(u)=I(x)\},\label{pattern space representation}
    \end{align}
see Appendix \ref{Pattern space Appendix} or \cite{schneider2022geometry}. Writing the basis of the pattern space in a matrix $U_x$, we have $\langle U_x\rangle=Im(U_x)$. 
Importantly, the pattern space does not depend on the smooth part $g(x)$ in (\ref{general linear framework}). Moreover, for any $v_0\in\partial f(x)$, we have 
\begin{equation*}
    \partial f(x) - v_0  \subset \langle U_x\rangle^{\perp},
\end{equation*} 
and for any positive definite matrix $C$;
\begin{equation}\label{general orthogonality in subdifferentials}
     C^{-1/2}(\partial f(x)-v_0) \subset (C^{1/2}\langle U_{x} \rangle)^{\perp}.
\end{equation}
Additionally, if $g(x)=0$ in (\ref{general linear framework}), then
\begin{equation}\label{subdifferential as a face of the dual ball}
    \partial f(x)= \partial f(0) \cap (v_0+\langle U_x\rangle^{\perp}),
\end{equation}
for any $v_0\in\partial f(x)$, see Appendix \ref{Pattern space Appendix}.

In the rest of the article, we shall often make the following assumption and refer to it as assumption (\hyperref[assumption A]{A}):

\begin{assumptionA}\label{assumption A}
We shall say that a sequence of estimators $\hat{\beta}_n$ satisfies assumption (\hyperref[assumption A]{A}) if $\sqrt{n}(\hat{\beta}_n - \beta^0)$ converges weakly and weakly in pattern to the minimizer $\hat{u}$ of
\[
V(u) = \frac{1}{2}u^T C u - u^T W + f'(\beta^0; u),
\]
where \( W \) is some centered random vector, \( C \) is some positive definite matrix, and \( f \) is a convex penalty satisfying (\ref{general linear framework}).
\end{assumptionA}

\begin{theorem} \label{general pattern recovery}
Under the assumption (\hyperref[assumption A]{A}), probability of pattern recovery converges:
\begin{align*}
&\mathbb{P}\big[I(\hat{\beta}_n)=I(\beta^0)\big]\underset{n\rightarrow\infty}{\longrightarrow}  \mathbb{P} \big[\hat{u} \in \langle U_{\beta^0} \rangle\big] = \mathbb{P} \big[\zeta\in \partial f(\beta^0) \big],\\
    &\zeta = \mu + C^{1/2}(I-P)C^{-1/2}W,
\end{align*}
where $P$ is the projection onto $C^{1/2}\langle U_{\beta^0}\rangle$, $\mu=C^{1/2} P C^{-1/2}v_0$, and $v_0$ is any vector in $\partial f(\beta^0)$. In particular, if $W\sim\mathcal{N}(0,\sigma^2C)$, then
$\zeta\sim\mathcal{N}(\mu, \sigma^2 C^{1/2}(I-P)C^{1/2})$.

\end{theorem}

\begin{remark}\label{remark after general pattern recovery}
Explicitly, $P=C^{1/2}U_{\beta^0}(U_{\beta^0}^TCU_{\beta^0})^{-1}U_{\beta^0}^TC^{1/2}$ and $\langle P\rangle=C^{1/2}\langle U_{\beta^0}\rangle$. Throughout, we use $\langle \cdot \rangle$ to denote the column space of a matrix, and $\langle \cdot \rangle^{\perp}$ its orthogonal complement.  One can verify that the affine space \footnote{For $A\subset\mathbb{R}^p$ the affine space is defined as $\text{aff}(A)=span\{A-x_0\}-x_0$, where $x_0$ is any fixed vector in $A$.} of $\partial f(\beta^0)$ is
\begin{equation*}
    \zeta\in\text{aff}(\partial f(\beta^0))= v_0 + \langle U_{\beta^0}\rangle^{\perp}.
\end{equation*}
Indeed, both $\mu-v_0=C^{1/2} P C^{-1/2}v_0 -v_0$ and $\zeta-\mu=C^{1/2}(I-P)C^{-1/2}W$ are in $C^{1/2}\langle I-P\rangle=\langle U_{\beta^0}\rangle^{\perp}.$ 
Moreover, 
\begin{equation}\label{irrepresentability equivalence}
\mu=\mathbb{E}[\zeta]=C^{1/2} P C^{-1/2}v_0 \in C\langle U_{\beta^0}\rangle\cap \text{aff}(\partial f(\beta^0)),
\end{equation}
 which does not depend on the choice of $v_0\in\partial f(\beta^0)$. Also, if $g(x)=0$ in (\ref{general linear framework}), then the limiting event in Theorem \ref{general pattern recovery} is $\{\zeta\in\partial f(0)\}$ by (\ref{subdifferential as a face of the dual ball}). 
\end{remark}

\begin{proof}
By Corollary \ref{asymptotic sampling corollary general} and (\ref{pattern space representation}), we obtain:
\begin{equation*}
    \mathbb{P}[I(\hat{\beta}_n)=I(\beta^0)] \longrightarrow \mathbb{P}[I_{\beta^0}(\hat{u})=I(\beta^0)]=\mathbb{P} \big[\hat{u} \in \langle U_{\beta^0} \rangle\big].
\end{equation*}
Moreover, $\forall u\in\langle U_{\beta^0} \rangle$; $\partial f'({\beta^0};u) = \partial f(I_{\beta^0}(u)) = \partial f(\beta^0) $, since $I_{\beta^0}(u)=I(\beta^0)$ by (\ref{pattern space representation}). Consequently, the optimality condition (\ref{main optimality condition}) $W \in Cu + \partial f'({\beta^0};u)$ yields:
\begin{align}
        \hat{u} \in \langle U_{\beta^0} \rangle &\iff W \in C\langle U_{\beta^0} \rangle + \partial f(\beta^0)\nonumber\\
        & \iff C^{-1/2}W \in C^{1/2}\langle U_{\beta^0} \rangle + C^{-1/2}\partial f(\beta^0)\nonumber\\
        &\iff \underbrace{-C^{-1/2}v_0 + C^{-1/2}W}_{=:Y} \in \underbrace{C^{1/2}\langle U_{\beta^0} \rangle}_{=\langle P \rangle} + \underbrace{C^{-1/2}(\partial f(\beta^0)-v_0)}_{\subset\langle I-P\rangle}.\label{general equation which needs orthogonal decomposition}
\end{align}
We have $C^{1/2}\langle U_{\beta^0} \rangle = \langle P \rangle$ and by (\ref{general orthogonality in subdifferentials}) $C^{-1/2}(\partial f(\beta^0)-v_0)\subset \langle P\rangle^{\perp} = \langle I-P\rangle$. 
Decomposing $Y = PY +(I-P)Y$, (\ref{general equation which needs orthogonal decomposition}) reduces to $(I-P)Y\in C^{-1/2}(\partial f(\beta^0)-v_0)$. Thus (\ref{general equation which needs orthogonal decomposition}) yields
\begin{align}
 \hat{u} \in \langle U_{\beta^0} \rangle&\iff (I-P)Y\in C^{-1/2}(\partial
f(\beta^0)-v_0)\nonumber\\
&\iff v_0 + C^{1/2} (I-P)(-C^{-1/2}v_0 + C^{-1/2}W)\in \partial
f(\beta^0)\nonumber\\
& \iff C^{1/2} P C^{-1/2}v_0 + C^{1/2}(I-P)C^{-1/2}W \in \partial
f(\beta^0),\label{zeta decomposition}
\end{align}
and using that $W\sim\mathcal{N}(0,\sigma^2 C)$, the above Gaussian vector has expectation $C^{1/2}P C^{-1/2}v_0$ and covariance matrix $\sigma^2 C^{1/2}(I-P)C^{1/2}$, which finishes the proof.
\end{proof}
Observe that Theorem \ref{general pattern recovery} is based on the equivalence
\begin{equation*}
    \hat{u} \in \langle U_{\beta^0} \rangle \iff W \in C\langle U_{\beta^0} \rangle + \partial f(\beta^0) \iff \zeta\in \partial f(\beta^0).
\end{equation*}
Moreover, Theorem \ref{general pattern recovery} reveals when it is possible to recover the true pattern with high probability as the penalization increases. Indeed, pattern recovery is possible if and only if $\mathbb{E}[\zeta]\in ri(\partial f(\beta^0))$, where $ri (\partial f(\beta^0))$ is the relative interior \footnote{For $A\subset\mathbb{R}^p$, $ri(A)$ is the interior of $A$ in \text{aff(A)}, where \text{aff(A)}$\subset \mathbb{R}^p$ is equipped with the subset topology. } of $\partial f(\beta^0)$ w.r.t. the affine space $\text{aff}(\partial f(\beta^0))= v_0 + \langle U_{\beta^0}\rangle^{\perp}$. We can view this as the asymptotic irrepresentability condition, which explicitly reads:
\begin{equation}\label{irrepresentability condition}
    C^{1/2}P C^{-1/2}v_0 \in ri (\partial f(\beta^0)),
\end{equation}
where $P$ is the projection onto $C^{1/2}\langle U_{\beta^0}\rangle$ and $v_0\in\partial f(\beta^0)$. Or equivalently, 
\begin{equation}\label{irrepresentability condition compact}
    0 \in (I-P)C^{-1/2}ri(\partial f(\beta^0)).
\end{equation}
Alternatively, by (\ref{irrepresentability equivalence}), the irrepresentability condition can be formulated equivalently as
\begin{equation}\label{irrepresentability condition geometric}
C\langle U_{\beta^0}\rangle\cap ri(\partial f(\beta^0))\neq\emptyset.
\end{equation}
Closely related versions of this condition for general penalties are explored in detail in \cite{vaiter2017model},\cite{graczyk2023pattern}, and for SLOPE in \cite{bogdan2022pattern}. For Lasso, (\ref{irrepresentability condition}) reduces to the Lasso irrepresentability condition \cite{zhao2006model}.

Consequently, if (\ref{irrepresentability condition}) holds, the probability of limiting pattern recovery converges to one as the penalty scaling increases. More precisely:
\begin{corollary}\label{pattern recovery exponential bound}
     Let $f_n=n^{1/2}\alpha f$, where $f$ is some fixed penalty function of the form (\ref{general linear framework}). Assume that the asymptotic irrepresentability condition (\ref{irrepresentability condition}) holds, and that the vector $W$ in (\hyperref[assumption A]{A}) has sub-gaussian entries. Then
    \begin{align*}\lim_{n\rightarrow\infty}\mathbb{P}[I(\hat{\beta}_n)=I(\beta^0)]
    \geq 1- 2e^{-c\alpha^2},
    \end{align*}
    for some positive constant $c$.
\end{corollary}
\begin{proof} 
By Theorem \ref{general pattern recovery}, $\mathbb{P}[I(\hat{\beta}_n)=I(\beta^0)]$ converges to $\mathbb{P}[\alpha\mu + BW \in \alpha \partial f(\beta^0)]$, where $B=C^{1/2}(I-P)C^{-1/2}$ and $\mu \in  ri(\partial f(\beta^0))$ by (\ref{irrepresentability condition}). Let $d>0$ denote the distance between $\mu$ and the boundary of $\partial f(\beta^0)$. Then
\begin{align*}
\lim_{n\rightarrow\infty}\mathbb{P}[I(\hat{\beta}_n)\neq I(\beta^0)]&=\mathbb{P}[BW\notin\alpha(\partial f(\beta^0)-\mu))] \\
&\leq \mathbb{P}\left[ \Vert BW\Vert> \alpha d \right]\\
&\leq 2e^{-c\alpha^2},
\end{align*}
for some $c>0$. 
\end{proof}
For more exact sub-gaussian tail bounds we refer to the Hanson--Wright concentration inequality in Theorem 6.2.1 or Theorem 6.3.2 \cite{vershynin2018high}.

\subsection{Two-step recovery}
Exact pattern recovery can be obtained for an arbitrary covariance structure $C$ employing the two-step proximal method (\ref{two-step procedure}) described in this section. This idea has already been used for SLOPE in \cite{graczyk2023pattern}. Here we develop new theory and prove model consistency of the second-order method for regularizers of the form (\ref{general linear framework}). 

Observe that for $C=\mathbb{I}$, the irrepresentability condition (\ref{irrepresentability condition}) will always be satisfied, provided that the pattern space $\langle U_{\beta^0}\rangle$ intersects the relative interior of $\partial f(\beta^0)$. This is satisfied for the Lasso, SLOPE, or the Concavified Fused Lasso (see Proposition \ref{concavification of Fused Lasso}) and these methods will recover their respective model patterns in the sense of Corollary \ref{pattern recovery exponential bound}, when $C=\mathbb{I}$. However, if $C\neq\mathbb{I}$, the aforementioned first-order methods will fail to recover their pattern with high probability if $C\langle U_{\beta^0}\rangle\cap ri(\partial f(\beta^0))\neq\emptyset$. The problem of strong covariates can be addressed by higher-order methods.
    
    For a convex penalty $f:\mathbb{R}^p\rightarrow \mathbb{R}$, the proximal operator is defined as the map from $\mathbb{R}^p$ to $\mathbb{R}$ given by \begin{equation*}
\text{Prox}_{f}(\beta):=\underset{\xi\in\mathbb{R}^p}{\operatorname{argmin}}\hspace{0.2cm}\frac{1}{2}\Vert \beta-\xi\Vert_2^2+f(\xi).
    \end{equation*}
The two-step procedure consists of:
    \begin{align}\label{two-step procedure}
        \text{Step 1: }& \text{Obtaining an initial estimate } \hat{\beta}^{(1)} \text{ of } \beta^0. \nonumber \\
        \text{Step 2: }& \text{Obtaining a truncated estimate }\hat{\beta}^{(2)}=\text{Prox}_{f}(\hat{\beta}^{(1)}). 
    \end{align}
The truncated estimate $\hat{\beta}^{(2)}$ is designed to recover the $f-$ pattern of the signal $\beta^0$, see Theorem~\ref{two-step proximal method theorem}. It can be heavily biased and therefore does not produce an accurate estimate of the signal in terms of MSE. The estimate of the pattern $M=I(\hat{\beta})$ after Step 2. can be used to obtain an asymptotically unbiased estimate $\hat{\beta}^{(3)}=\hat{\beta}_{OLS(M)}$ of $\beta^0$:
\begin{align}\label{three-step procedure}
    \text{Step 3: } \hat{\beta}^{(3)}=U_M(X_M^TX_M)^{-1}X_M^Ty,
\end{align}
where $X_M=X U_M$, and $U_M=(b_1,\dots,b_{\vert M\vert})$ is any fixed basis \footnote{The estimate $\hat{\beta}_{OLS(M)}$ does not depend on the choice of basis $U_M$ of $\langle U_M\rangle_f$. Indeed, for any other basis $\tilde{U}_M=(\tilde{b}_1,\dots,\tilde{b}_{\vert M\vert})$, there is an invertible matrix $Q\in\mathbb{R}^{\vert M\vert \times \vert M\vert}$, such that $\tilde{U}_M=U_M Q$. For $\tilde{X}_M=X\tilde{U}_M=X_M Q$, we have $\tilde{U}_M(\tilde{X}_M^T\tilde{X}_M)^{-1}\tilde{X}_M^T=U_M(X_M^TX_M)^{-1}X_M^T$.} of the pattern space $\langle U_M\rangle_f$. If the true pattern is recovered after Step 2, i.e. $M=I(\hat{\beta})=I(\beta_0)$, then $\hat{\beta}^{(3)}=\hat{\beta}_{OLS(M)}$ is unbiased \footnote{If $M=I(\beta^0)$, then $\beta^0=U_M \beta_M$ for some $\beta_M\in\mathbb{R}^{\vert M\vert}$. The linear model then reduces to $y=X_M \beta_M + \varepsilon$ and $\mathbb{E}[\hat{\beta}_{OLS(M)}]=U_M\beta_M=\beta^0$ by (\ref{three-step procedure}).} for $\beta^0$. The true pattern can be recovered even in the high-dimensional regime when $p>n$, but in particular the correct pattern is recovered after Step 2 with high probability for fixed $p$ and $n\rightarrow\infty$, see Theorem \ref{two-step proximal method theorem}. Consequently, the 3 Step procedure is asymptotically unbiased (with high probability). The third step is possible if the reduced design matrix $X_M$ has full rank, for which $\vert M\vert\leq p$ is necessary.
\begin{lemma}\label{two-step proximal method lemma}
Let $f$ be a convex penalty of the form (\ref{general linear framework}) and $\hat{\beta}^{(1)}_n$ a sequence of estimators such that $\sqrt{n}(\hat{\beta}^{(1)}_n-\beta^0)\overset{d}{\longrightarrow} W$ for some random vector $W$. Let $\hat{\beta}^{(2)}_n=\text{Prox}_{n^{-1/2}f}(\hat{\beta}^{(1)}_n)$, i.e. the minimizer of 
\begin{equation*}
    M_n(\beta):= \dfrac{1}{2}\Vert \hat{\beta}^{(1)}_n-\beta\Vert_2^2 + n^{-1/2}f(\beta).
\end{equation*}
 Then $\sqrt{n}(\hat{\beta}^{(2)}_n-\beta^0)$ converges weakly and weakly in pattern to the minimizer $\hat{u}$ of: 
\begin{equation*}
    V(u)=\dfrac{1}{2}\Vert u\Vert_2^2 - u^T W + f'(\beta^0;u).
\end{equation*}
\end{lemma}
\begin{proof} The error $\hat{u}_n=\sqrt{n}(\hat{\beta}^{(2)}_n-\beta^0)$ minimizes
    \begin{align*}
    V_n(u)&=n(M_n(\beta^0+u/\sqrt{n})-M_n(\beta^0))\\
    &= n\left(\dfrac{1}{2}\Vert( \hat{\beta}^{(1)}_n-\beta^0)-u/\sqrt{n}\Vert_2^2 -\dfrac{1}{2}\Vert \hat{\beta}^{(1)}_n-\beta^0\Vert_2^2\right) + \sqrt{n}(f(\beta^0+u/\sqrt{n})-f(\beta^0))\\
    &=\dfrac{1}{2}\Vert u\Vert_2^2 - u^T\sqrt{n}(\hat{\beta}^{(1)}_n-\beta^0) + \sqrt{n}(f(\beta^0+u/\sqrt{n})-f(\beta^0))\\
    &\overset{d}{\longrightarrow} \dfrac{1}{2}\Vert u\Vert_2^2 - u^T W +f'(\beta^0;u).
    \end{align*}
Consequently, convergence in distribution follows by the convexity of the objectives as in Theorem \ref{sqrt-asymptotic}. Weak pattern convergence follows as in Theorem \ref{main theorem pattern convergence of convex sets}.
\end{proof}
\begin{example}
For $\hat{\beta}^{(1)}_n$ equal to the OLS estimator, $\sqrt{n}(\hat{\beta}^{(1)}_n-\beta^0)\overset{d}{\longrightarrow}\mathcal{N}(0, \sigma^2 C^{-1})$. For the Ridge estimator with penalty sequence $\alpha_n/\sqrt{n}\rightarrow\alpha\geq 0$, $\sqrt{n}(\hat{\beta}^{(1)}_n-\beta^0)\overset{d}{\longrightarrow}W\sim\mathcal{N}(-\alpha C^{-1}\beta^0,\sigma^2C^{-1})$.
\end{example}

\begin{theorem}\label{two-step proximal method theorem}
Let $f$ be of the form (\ref{general linear framework}) and $\hat{\beta}^{(1)}_n$ such that $\sqrt{n}(\hat{\beta}^{(1)}_n-\beta^0)\overset{d}{\longrightarrow} W$ for some random vector $W$ with sub-gaussian entries. Then for $\hat{\beta}^{(2)}_n=\text{Prox}_{n^{-1/2}\alpha f}(\hat{\beta}^{(1)}_n)$, for all $\alpha\geq0$,
\begin{align*}
\lim_{n\rightarrow\infty}\mathbb{P}\big[I(\hat{\beta}^{(2)}_n)=I(\beta^0)\big]\geq 1- 2e^{-c\alpha^2},
\end{align*}
for some $c>0$, provided the pattern space $\langle U_{\beta^0}\rangle$ intersects the relative interior of $\partial f(\beta^0)$.

\end{theorem}
\begin{proof}
From Lemma~\ref{two-step proximal method lemma} and Theorem \ref{general pattern recovery} (with $C=\mathbb{I}$), we obtain
\begin{align*}
\lim_{n\rightarrow\infty}\mathbb{P}\big[I(\hat{\beta}^{(2)}_n)=I(\beta^0)\big]&=\mathbb{P}[\alpha \mu+(I-P_{\beta^0})W\in\alpha\partial f(\beta^0)],
\end{align*}
where $P_{\beta^0}$ is a projection onto the pattern space $\langle U_{\beta^0}\rangle$, $\mu=P_{\beta^0}v_0$, $v_0\in\partial f(\beta^0)$. By assumption, (\ref{irrepresentability condition}) holds for $C=\mathbb{I}$, hence $\mu \in ri(\partial f(\beta^0))$. The bound follows by sub-gaussianity of $W$ as in Corollary \ref{pattern recovery exponential bound}.
\end{proof}

Whether the two step-proximal method asymptotically recovers the corresponding pattern of $\beta^0$ w.h.p. as $\alpha$ increases does not depend on the covariance structure $C$, but only on the condition $\langle U_{\beta^0}\rangle\cap ri(\partial f(\beta^0))\neq\emptyset$ in Theorem~\ref{two-step proximal method theorem}. For Lasso and SLOPE the condition is satisfied for every signal vector $\beta^0$, hence the two-step proximal method based on these penalties asymptotically recovers the respective pattern of any $\beta^0$, w.h.p. as $\alpha$ increases. The recovery holds for any covariance structure $C$. Interestingly, for Fused Lasso, there are signals $\beta^0$, for which $\langle U_{\beta^0}\rangle\cap ri(\partial f(\beta^0))=\emptyset$ (see Figure \ref{fig: Irrepresentability for Fused Lasso} and Example \ref{Fused Lasso counterexample}). Patterns of such signal vectors will not be recovered by the two-step proximal method. In Proposition \ref{concavification of Fused Lasso}, we show how this problem can be solved by a small modification to the Fused Lasso penalty.

\subsection{Pattern attainability}
We can also characterize all patterns $\mathfrak{p}\in\mathfrak{P}$, for which $\mathbb{P}[I(\hat{u})=\mathfrak{p}]>0$, in terms of the pattern space $\langle U_{\mathfrak{p}}\rangle=span\{u: I(u)= \mathfrak{p}\}$. In the following Proposition we assume (\hyperref[assumption A]{A}), and that $W$ has a density with respect to the Lebesgue measure.
\begin{proposition}\label{attainability proposition}
    For a pattern $\mathfrak{p}\in\mathfrak{P}$, $\mathbb{P}[I(\hat{u})=\mathfrak{p}]>0$ if and only if the pattern spaces of $\mathfrak{p}$ and $\mathfrak{q}=I_{\beta^0}(\mathfrak{p})$ coincide, i.e. $\langle U_{\mathfrak{p}}\rangle=\langle U_{\mathfrak{q}}\rangle$. This is equivalent to $dim(\partial f(\mathfrak{q}))=dim(\partial f(\mathfrak{p}))$.
\end{proposition}
In particular, there is always a positive probability of recovering the true pattern, i.e. $\mathbb{P}[I(\hat{u})=I(\beta^0)]>0$, because for $\mathfrak{p}=I(\beta^0)$, also $\mathfrak{q}=I_{\beta^0}(\mathfrak{p})=I(\beta^0)$. 

\section{Examples}
We discuss several examples that fit into the framework (\ref{general linear framework}) and for which the theoretical results from the previous section can be applied.
\subsection{Generalized Lasso}
\begin{example}
Lasso penalty can be written in the form (\ref{general linear framework}) as $f(x)=\max\{\langle S \boldsymbol{\lambda},x\rangle: S\in\mathcal{S}\}$, where $\boldsymbol{\lambda}=(\lambda,\dots,\lambda)^T\in\mathbb{R}^p$ with $\lambda>0$ 
and $\mathcal{S}$ consists of $2^p$ diagonal matrices with entries $+1$ or $-1$. The Lasso pattern $I(x)\subset\mathcal{S}$ can be identified with the sign $I(x)\cong sgn(x)$ in the sense that $I(x)=I(y) \iff sgn(x)=sgn(y)$. There are $3^p$ distinct patterns in $\mathfrak{P}$. The subdifferential $\partial f(x)=con\{S\boldsymbol{\lambda}: S\in I(x)=\operatorname{argmax}_{S\in\mathcal{S}}\langle S\boldsymbol{\lambda}, x\rangle\}$ can be written as the Cartesian product of singletons $sgn(x_i)\lambda$ for $x_i\neq0$ and closed intervals $[-\lambda, \lambda]$ for $x_i=0$. 
\end{example}
\begin{example} 
Generalized Lasso reads $f_A(x)=\lambda\Vert Ax\Vert_1=\max\{\langle A^TS\boldsymbol{\lambda}, x \rangle : S\in\mathcal{S}\}$, where $A$ is an $m\times p$ matrix. The pattern can be identified with $I_A(x)\cong sgn(Ax)$, and the subdifferential is $\partial f_A(x)=A^T\partial f(Ax)=A^Tcon\{ S\boldsymbol{\lambda}: S\in \operatorname{argmax}_{S\in\mathcal{S}}\langle S\boldsymbol{\lambda}, Ax\rangle\}$, where $f$ is the standard Lasso penalty from the previous example. This follows from (\ref{composition of patterns}) with $\psi(x)=Ax$, $f(x)=\lambda\Vert x\Vert_1$, $f_A = f\circ\psi$. 
\end{example}

\begin{example}\label{Fused Lasso counterexample}(Fused Lasso)
Here we illustrate how the Fused Lasso fails to asymptotically recover its own patterns, even when $C=I$. Let $f_A(\beta)=\lambda\Vert A\beta\Vert_1$, $\lambda>0$. Let $\beta^0=(1,2,2,3)^T$, and for $a_1, a_2, a_3>0$, consider 

\begin{align*}
    A=\left[\begin{matrix}
a_1 & -a_1 & 0 & 0 \\
0 & a_2 & -a_2 & 0 \\
0 & 0 & a_3 & -a_3 \\
\end{matrix} \right] 
\hspace{1cm} 
\partial f_A(\beta^0)=
\lambda\begin{pmatrix} -&a_1 \\ &a_1 \\ -&a_3 \\ &a_3 \end{pmatrix}+
\lambda\left(\begin{array}{c}
0\\
con \left\{\left(\begin{matrix} a_2 \\ -a_2 \end{matrix}\right), \left(\begin{matrix} -a_2 \\ a_2 \end{matrix}\right)\right\}\\
0
\end{array}
\right),
\end{align*}
where the subdifferential is computed as $\partial f_A(\beta^0)= con\{A^T S \boldsymbol{\lambda}: S\in I_A(\beta^0)
\}$, where the pattern $I_A(\beta^0)=\operatorname{argmax}_{S\in\mathcal{S}}\langle  S \boldsymbol{\lambda},A\beta^0 \rangle\}=diag((-1,\{\pm1\},-1))$ consists of two diagonal matrices. The pattern space $\langle U_{\beta^0}\rangle$ is spanned by all vectors $\beta$ such that $I_A(\beta)=I_A(\beta^0)$. Explicitly, $\langle U_{\beta^0}\rangle = span\{\beta: sgn(A\beta)=sgn(A\beta^0)\}=span\{(1,0,0,0)^T, (0,1,1,0)^T, (0,0,0,1)^T\}$. For the case where $C=I$, (\ref{irrepresentability condition compact}) becomes $0\in (I-P)ri(\partial f_A(\beta^0))$, where $P=U_{\beta^0}(U_{\beta^0}^TU_{\beta^0})^{-1}U_{\beta^0}^T$ is the projection on $\langle U_{\beta^0} \rangle$. 

\begin{align*}
    P=\left[\begin{matrix}
1 & 0 & 0 & 0 \\
0 & 1/2 & 1/2 & 0 \\
0 & 1/2 & 1/2 & 0 \\
0 & 0 & 0 & 1
\end{matrix} \right] 
\hspace{1cm} 
(I-P)\partial f_A(\beta^0)=
\lambda\left(\begin{array}{c}
0\\
\begin{matrix} -(a_1+a_3)/2 \\ \;\;\;(a_1+a_3)/2 \end{matrix} + con \left\{\pm\left(\begin{matrix} a_2 \\ -a_2 \end{matrix}\right)\right\}\\
0
\end{array}
\right),
\end{align*}
We see that the irrepresentability condition is satisfied iff $a_1/2+a_3/2=\gamma a_2 + (1-\gamma) a_2$ for some $\gamma$ with $\vert\gamma\vert<1$, which is equivalent to $a_1/2+a_3/2<a_2$. For the standard tuning for Fused Lasso $a_1=a_2=a_3=1$, therefore $a_1/2+a_3/2=a_2$ and (\ref{irrepresentability condition compact}) does not hold. Consequently, the probability that the standard Fused Lasso (with $a_1=a_2=a_3$) recovers the pattern of $\beta^0=(1,2,2,3)$ is bounded by $1/2$ as $n\rightarrow\infty$, for any penalty $\lambda>0$. On the other hand, if the triple $(a_1, a_2, a_3)$ is strictly concave, the pattern of $\beta^0$ will be recovered with high probability. Surprisingly, a slight concavification of the clustering penalties rectifies the asymptotic recovery for all patterns. The following proposition asserts this result. For proof we refer to the Appendix \ref{apppendix proofs}.
\end{example}
\begin{proposition}[Concavification of Fused Lasso]\label{concavification of Fused Lasso}
    For $C=\mathbb{I}$, the (tuned) Fused Lasso $f_A(\beta)=\lambda\Vert A\beta\Vert_1 = \lambda\sum_{i=1}^{p-1} a_i\vert \beta_{i+1}-\beta_i\vert + \lambda\sum_{i=1}^p a\vert \beta_i\vert$, $a_i>0\;\forall i$, $a, \lambda>0$, asymptotically recovers all its patterns, i.e.;
    \begin{equation*}
            \forall \beta^0\in\mathbb{R}^{p}; \;\;\;\lim_{n\rightarrow\infty}\mathbb{P}[I_A({\hat{\beta}}_n)=I_A(\beta^0)]\underset{\lambda\rightarrow\infty}{\longrightarrow} 1, 
    \end{equation*}
    if and only if $(0, a_1, \dots, a_{p-1},0)$ forms a strictly concave sequence \footnote{This means there exists a strictly concave function $F:\mathbb{R}\rightarrow\mathbb{R}$, such that $a_i=F(i)$, for $i=0, 1\dots,p$.} and the sparsity penalty $a>\max\{a_i +a_{i+1} : 0\leq i\leq p-1\}$, where we set $a_0=a_p=0$.
\end{proposition}
For $p=1$, $f_A(\beta)=\lambda\vert a\beta\vert$, $\beta\in\mathbb{R}$, the conditions in the Proposition reduce to $a>a_0+a_1=0$. 
We see that if $a>0$, $\beta^0=0$ will be recovered by $\hat{\beta}_n$ as $\lambda$ increases. If $\beta^0\neq 0$, then because $\hat{\beta}_n\rightarrow \beta^0$ in probability (recall that $\sqrt{n}(\hat{\beta}_n-\beta^0)=O_p(1)$ by Theorem \ref{sqrt-asymptotic}), it follows that $\mathbb{P}[sgn(\hat{\beta}_n)=sgn(\beta^0)]$ goes to one as $n\rightarrow\infty$. Conversely, if $a=0$, there will be no shrinkage, and $\beta^0=0$ will not be recovered by $\hat{\beta}_n$.

For $p=2$, $f_A(\beta)=\lambda(a_1\vert\beta_2-\beta_1\vert + a\vert \beta_1\vert+a\vert \beta_2\vert)$, the above conditions read $a>a_1$. Now all patterns are recoverable if and only if (\ref{irrepresentability condition geometric}) holds, i.e. $\langle U_{\beta^0}\rangle \cap ri(\partial f_A(\beta^0))\neq \emptyset$  for every $\beta^0$.  
\begin{figure}[ht]
    \centering
    \begin{tikzpicture}[scale=1]
        \draw[fill = gray!20!white] (1+1,-1+1) -- (1+1,-1-1) -- (1-1,-1-1) -- (-1-1,-1+1) -- (-1-1,1+1) -- (-1+1,1+1) -- cycle;
        \draw[gray!90!white] (1+1.2,-1+1.2) -- (1+1.2,-1-1.2) -- (1-1.2,-1-1.2) -- (-1-1.2,1-1.2) -- (-1-1.2,1+1.2) -- (-1+1.2,1+1.2) -- cycle;
        \draw[-latex,dotted] (-3,0) -- (3,0) node[below] {};
        \draw[-latex, dotted] (0,-3) -- (0,3) node[left] {};
        \node at (1.7,1.7){$\partial f_{A}(0)$};
        \draw[green, line width=1pt] (2,-2) -- (2,0) node[left, green!50!black] at (2.1,-0.9){$ \partial f_{A}(\beta^0)$};
        \draw[green, line width=1pt] (2.2,-2.2) -- (2.2,0.2) ;
        \draw[dotted, purple!90!black, line width=1.4pt] (-2.6,0) -- (2.6,0) node[below] at (-2.6, 0) {$\langle U_{\beta^0} \rangle$};

        \draw[dotted, line width=1pt] (-1,1) -- (-1,2.2) node[left, black] at (-1,1.5){$a$};
        \draw[dotted, line width=1pt] (-1,1) -- (0,1) node[below, black] at (-0.5,1){$a_1$};
        
        
    \end{tikzpicture}
    \caption{Asymptotic irrepresentability condition $\langle U_{\beta^0}\rangle \cap ri(\partial f_A(\beta^0))\neq \emptyset\;\forall \beta^0 \iff a>a_1$.}
    \label{fig: Irrepresentability for Fused Lasso}
\end{figure}
Geometrically, Figure \ref{fig: Irrepresentability for Fused Lasso} illustrates that recovery of all patterns is possible if and only if $a>a_1$. We see that when $a\leq a_1$, $\langle U_{\beta^0}\rangle \cap ri(\partial f_A(\beta^0))= \emptyset$, the pattern of $\beta^0=(1,0)^T$ will not be recovered with high probability.
\subsection{SLOPE}
The SLOPE norm \cite{bogdan2015slope} (resp. OSCAR \cite{Bondell2008oscar}, OWL\cite{OWL2}) is defined through a non-increasing sequence $\lambda_1\geq\dots\geq\lambda_p\geq0$,
\begin{equation*}
    J_{\boldsymbol{\lambda}}(\beta):= \sum_{i=1}^{p}\lambda_i\vert\beta\vert_{(i)},
\end{equation*}
where $\vert \beta\vert_{(\cdot)}$ is the order statistic of $(\vert \beta_1\vert,\dots,\vert \beta_p\vert)$, i.e., $\vert\beta\vert_{(1)}\geq\dots\geq\vert\beta\vert_{(p)}$. 

The SLOPE penalty can be recast as $J_{\boldsymbol{\lambda}}(x) = \max\{ \langle P\boldsymbol{\lambda},x\rangle: P\in\mathcal{S}_p^{+/-}\}$, where $\mathcal{S}_p^{+/-}$ is the set of signed permutation matrices. If $\lambda_1>\dots>\lambda_p>0$, the pattern $I(x)\subset\mathcal{S}^{+/-}_p$ can be identified with $I(x)\cong \mathbf{patt}(x):=rank(\vert x_i\vert)sgn(x_i)$, since
$\mathbf{patt}(x)=\mathbf{patt}(y)$ if and only if $\partial J_{\boldsymbol{\lambda}}(x)=\partial J_{\boldsymbol{\lambda}}(y)$, see \cite{schneider2022geometry}.
For penalty vectors $\boldsymbol{\lambda}$, which are not strictly decreasing, the set of all patterns $\mathfrak{P}$ contains fewer elements, and the identification $I(x)\cong \mathbf{patt}(x)$ no longer holds. The subdifferential 
\begin{equation*}
    \partial J_{\boldsymbol{\lambda}}(x)= con\{P\boldsymbol{\lambda}: P\in I(x)\subset \mathcal{S}^{+/-}_p\},
\end{equation*}
is described more explicitly in Appendix \ref{SLOPE subdifferential Appendix}.  The subdifferential of the SLOPE norm has already been explored in \cite{bu2019algorithmic,bu2020algorithmic,larsson2020strong, schneider2022geometry, tardivel2020simple}. 
We refer the reader to \cite{bu2019algorithmic} for further details about the representation of the SLOPE subdifferential in terms of Birkhoff polytopes and to \cite{schneider2022geometry} and \cite{larsson2020strong} for different derivations of the SLOPE subdifferential.

The directional derivative $f'(\beta^0; u) = J'_{\boldsymbol{\lambda}}(\beta^0; u)$ is given by 
\begin{equation}
    J'_{\boldsymbol{\lambda}}({\beta^0};u)=\sum\limits_{i=1}^p\lambda_{\pi(i)}\left[u_i sgn(\beta^0_i)\mathbb{I}[\beta^0_i\neq0]+\vert u_i\vert\mathbb{I}[\beta^0_i=0]\right]\label{directional SLOPE derivative},
\end{equation}
where $\pi$ is a permutation which sorts the vector $\vert \beta^0+\varepsilon u\vert=(\vert \beta^0_1+\varepsilon u_1\vert, \dots, \vert \beta^0_p+\varepsilon u_p\vert )$ for $\varepsilon>0$ sufficiently small \footnote{This means that $ \vert \beta^0+\varepsilon u\vert_{\pi^{-1}(1)}\geq...\geq\vert \beta^0+\varepsilon u\vert_{\pi^{-1}(p)}$ as $\varepsilon\downarrow 0$.}, for derivation, see Appendix \ref{appendix directional derivative for SLOPE}. Note that the Lasso directional derivative, described in \cite{fu2000asymptotics}, is a special case of (\ref{directional SLOPE derivative}), where the permutation $\pi$ is omitted. 

In the context of SLOPE, as a consequence of Theorem \ref{main theorem pattern convergence of convex sets} and Corollary \ref{asymptotic sampling corollary general}, for any pattern $\mathfrak{p}\in\mathfrak{P}$ we have:
\begin{align}\label{SLOPE pattern asymptotics}
\mathbb{P}[\mathbf{patt}(\hat{u}_n)=\mathfrak{p}]&\xrightarrow[n\rightarrow\infty]{}\mathbb{P}[\mathbf{patt}(\hat{u})=\mathfrak{p}],\nonumber\\
\mathbb{P}[\mathbf{patt}(\hat{\beta}_n)=\mathfrak{p}]&\xrightarrow[n\rightarrow\infty]{}\mathbb{P}[\mathbf{patt}_{\beta^0}(\hat{u})=\mathfrak{p}],
\end{align}

where $\hat{u}$ minimizes (\ref{V(u)}) and $\mathbf{patt}_{\beta^0}(u)=\operatorname{lim}_{\varepsilon\downarrow0}\mathbf{patt}(\beta^0+\varepsilon u)$ denotes the limiting pattern. We note that (\ref{SLOPE pattern asymptotics}) remains valid even for a penalty vector $\boldsymbol{\lambda}$, which is not strictly decreasing, despite the fact that identification $I(x)\cong \mathbf{patt}(x)=rank(\vert x_i\vert)sgn(x_i)$ no longer holds. This follows from Theorem \ref{main theorem pattern convergence of convex sets}. In fact, (\ref{SLOPE pattern asymptotics}) holds for any sequence of penalties $\boldsymbol{\lambda}^n/\sqrt{n}\rightarrow\boldsymbol{\lambda}\geq0$.
A closed form expression for the limiting probability of pattern recovery for SLOPE has been described in Theorem 4.2 i) \cite{bogdan2022pattern}. The result also follows from Theorem \ref{general pattern recovery} and Remark \ref{remark after general pattern recovery}:
\begin{align*}
     &\mathbb{P}\big[\mathbf{patt}(\hat{\beta}_n)=\mathbf{patt}(\beta^0)\big]\underset{n\rightarrow\infty}{\longrightarrow}  \mathbb{P} \big[\zeta\in \partial J_{\boldsymbol{\lambda}}(\beta^0) \big]=\mathbb{P} \big[\zeta\in \partial J_{\boldsymbol{\lambda}}(0) \big],\\
    &\zeta\sim\mathcal{N}(C^{1/2}P C^{-1/2}\Lambda_0, \sigma^2 C^{1/2}(I-P)C^{1/2}),
\end{align*}
where $\Lambda_0\in \partial J_{\boldsymbol{\lambda}}(\beta^0)$ and $P$ is the projection matrix onto $C^{1/2}\langle U_{\beta^0}\rangle$. Explicitly, the pattern space $\langle U_{\beta^0}\rangle$ is spanned by the matrix $U_{\beta^0}=(\mathbf{1}_{I_m}\vert\dots\vert\mathbf{1}_{I_{1}})$, where $\{I_0, I_1, \dots,I_m\}$ is the corresponding partition of $\{1,\dots,p\}$ according to the clusters of $\beta^0$, and $\mathbf{1}_{I}\in\mathbb{R}^p$ the vector of ones supported on $I$. In the context of SLOPE, a cluster of $\beta^0$ is a subset $I\subset\{1,\dots,p\}$ such that $\vert \beta_i\vert=\vert \beta_j\vert$ for $i,j \in I$. Also, $\Lambda_0 = P_0\boldsymbol{\lambda}$, where $P_0$ is any matrix in $I(\beta^0)=\operatorname{argmax}_{P\in\mathcal{S}}\langle P\boldsymbol{\lambda},\beta^0\rangle$. For details, see Appendix \ref{SLOPE subdifferential Appendix}.
 
 \begin{example}
    We illustrate the results for the SLOPE norm $f(\beta)=J_{\boldsymbol{\lambda}}(\beta)$ with $\boldsymbol{\lambda}=(3,2)$. Let $\beta^0=(1,0)$, so that $\partial J_{\boldsymbol{\lambda}}(\beta^0)= con\{(3,2),(3,-2)\}$. The pattern matrix $U_{\beta^0}=(1,0)^T$ and $v_0=(3,2)\in\partial J_{\boldsymbol{\lambda}}(\beta^0)$. Let $C$ be unit diagonal with $\rho$ off diagonal. Condition (\ref{irrepresentability condition}) reads $C (3,0)^T=(3,3\rho)^T\in \partial J_{\boldsymbol{\lambda}}(\beta^0)$, or $\vert\rho\vert<2/3$. Let $\Lambda_1,\Lambda_2, \Lambda_3$ equal to $(3,3\rho)^T$ for $\rho=2/3+0.05, 2/3$ and $2/3-0.05$, respectively. Figure \ref{fig: Pattern recovery by SLOPE} shows that exact asymptotic pattern recovery is achieved if and only if the irrepresentability condition $\vert\rho\vert<2/3$ holds.

\begin{figure}[ht]
    \centering
    \begin{subfigure}[b]{0.48\textwidth} 
        \begin{tikzpicture}[scale=1]
            \draw[fill = gray!20!white] (3,-2) -- (3,2) -- (2,3) -- (-2,3) -- (-3,2)-- (-3,-2)-- (-2,-3)-- (2,-3) -- cycle;
            \draw[-latex,dotted] (-3.5,0) -- (3.8,0) node[below] {};
            \draw[-latex, dotted] (0,-3.5) -- (0,3.8) node[left] {};
            \node at (-1.2,1){$\partial J_{\boldsymbol{\lambda}}(0)$};
            \draw[green, line width=1.4pt] (3,-2) -- (3,2) node[left, green!50!black] at (3,-0.6){$\partial J_{\boldsymbol{\lambda}}(\beta^0)$};
            \draw[dotted, purple!90!black, line width=1.4pt] (3,-4) -- (3,4) node[left] at (3,-3.6) {$ \Lambda_0 + \langle U_{\beta^0} \rangle^{\perp}$};
            \filldraw[black] (3,-2) circle (2.2pt) node[right, black] (p0) at (3, -2) {$\Lambda_0$};
            \filldraw[red] (3,2.4) circle (2.2pt) node[above right, red] (p1) at (3, 2.4) {$\Lambda_1$};
            \filldraw[blue] (3,2) circle (2.2pt) node[right, blue] (p2) at (3, 2) {$\Lambda_2$};
            \filldraw[green!50!black] (3,1.6) circle (2.2pt) node[below right, green!50!black] (p3) at (3, 1.6) {$\Lambda_3$};
        \end{tikzpicture}
        \caption{Asymptotic irrepresentability condition satisfied for $\Lambda_3$ and violated for $\Lambda_1$ and $\Lambda_2$.}
        \label{fig:square_line}
    \end{subfigure}
    \hfill
    \begin{subfigure}[b]{0.48\textwidth} 
        \centering
        \includegraphics[width=\textwidth]{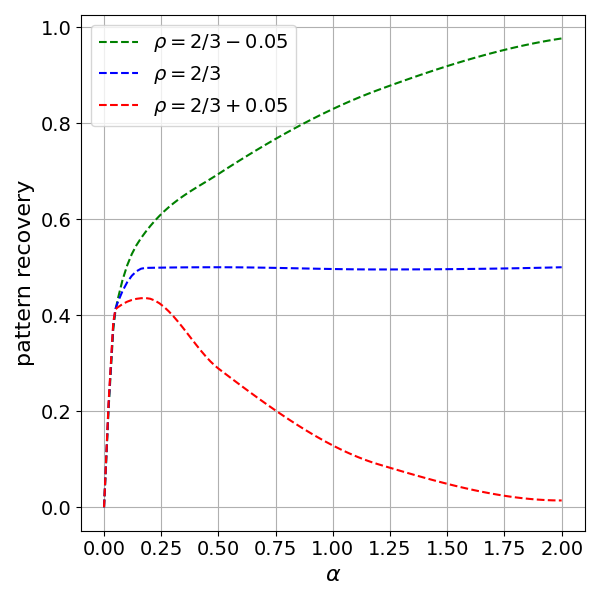}
        \caption{A phase transition in pattern recovery at $\rho=2/3$; $C=[[1,\rho],[\rho,1]]$, $\boldsymbol{\lambda}=\alpha[3,2], \sigma = 0.2$.}
        \label{fig:subfig_b}
    \end{subfigure}
    \caption{Asymptotic pattern recovery for SLOPE}
    \label{fig: Pattern recovery by SLOPE}
\end{figure}
\end{example}

\section{Simulations}\label{simulations}
We illustrate Theorem \ref{sqrt-asymptotic} in some simulations. We sample the asymptotic error $\hat{u}$, which minimizes $u^TCu/2-u^T W + \alpha f'(\beta^0;u)$, with $W\sim\mathcal{N}(0,\sigma^2C)$, $\alpha>0$. For Lasso and Fused Lasso, we use the ADMM algorithm and for SLOPE the proximal gradient descent \ref{proximal operator}. We compute the root mean squared error (RMSE) $(\mathbb{E}\Vert\hat{u}\Vert^{2})^{1/2}$ and the limiting probability of recovering the true pattern $\operatorname{lim}_{n\rightarrow\infty}\mathbb{P}[\mathbf{patt}(\hat{\beta}_n) = \mathbf{patt}(\beta^0)]$, (specifically, the exact SLOPE pattern).\footnote{The code for simulations can be found at https://github.com/IvanHejny/asymptotic-error-of-regularizers.git} We note that the distribution of $\hat{u}$ depends only on the pattern of $\beta^0$.   
\begin{figure}[bt]
      \centering
	     \begin{subfigure}{0.32\linewidth}
		 \includegraphics[width=\linewidth]{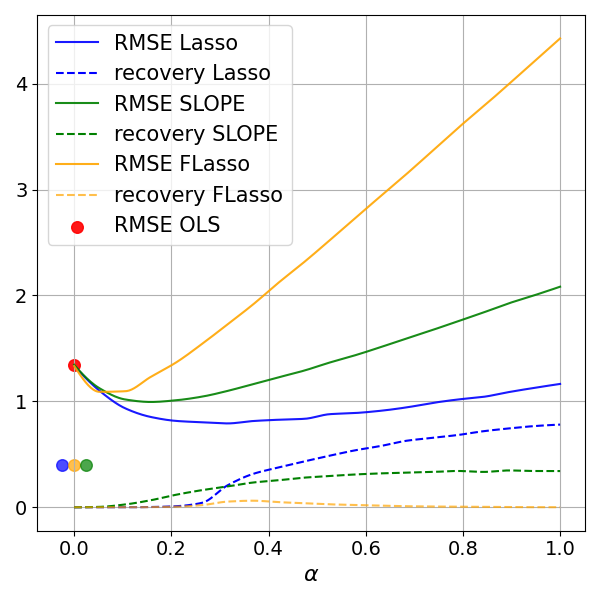}
		 \caption{  
        $\beta^0 = [0, 0, 1, 0] $}
	      \end{subfigure}
	      \begin{subfigure}{0.32\linewidth}
		  \includegraphics[width=\linewidth]{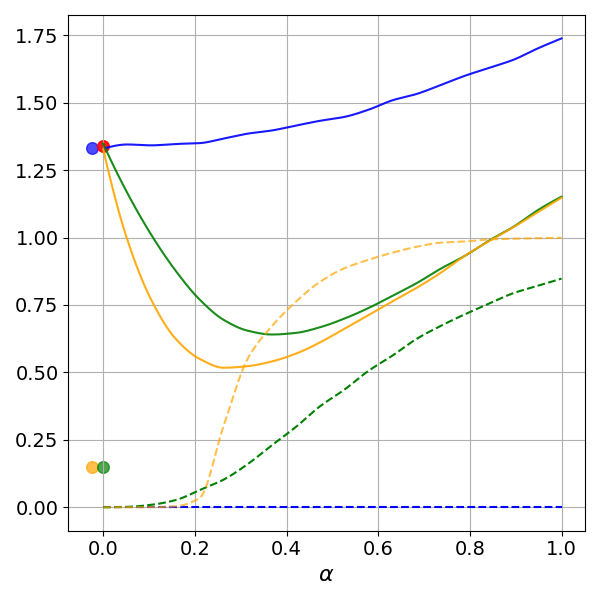}
		  \caption{  $
        \beta^0 = [1, 1, 1, 1]$}
	       \end{subfigure}
        \begin{subfigure}{0.32\linewidth}
		  \includegraphics[width=\linewidth]{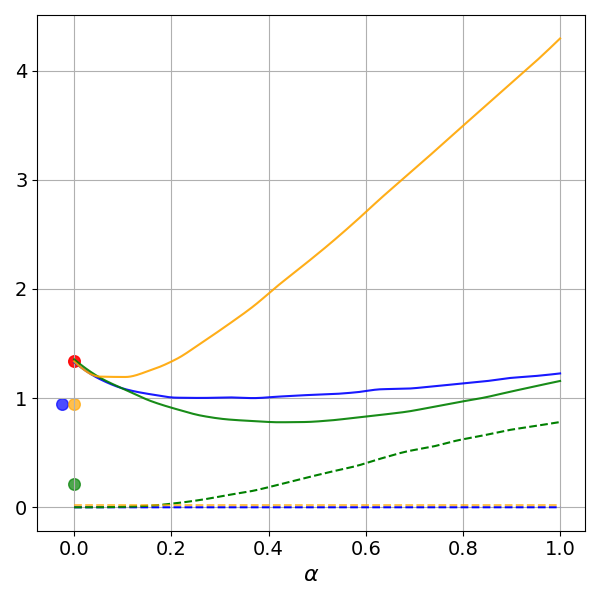}
		  \caption{  $
        \beta^0 = [1, 0, 1, 0]$}
	       \end{subfigure}
	\caption{Comparing root mean squared error (RMSE) for different methods together with the probability of pattern recovery, (i.e. correctly identifying all zeros and all clusters). }
	\label{fig: comparing RMSE}
\end{figure}

Figure \ref{fig: comparing RMSE} illustrates how performance depends on the pattern of the signal $\beta^0$. We consider a linearly decaying sequence as the penalty coefficients in SLOPE. This corresponds to the OSCAR sequence \cite{Bondell2008oscar}. In (a), Lasso best exploits the sparsity of $\beta^0$ and outperforms both SLOPE and Fused Lasso. In (b), Fused Lasso performs best, taking advantage of the consecutively clustered signal. Finally, in $c)$, SLOPE can discover clusters in nonneighboring coefficients, which the Fused Lasso cannot. In this situation, SLOPE has better estimation properties than the other methods.

Moreover, to showcase the strength of dimensionality reduction, we visualize the RMSE of the OLS in the reduced model, assuming perfect knowledge of the signal pattern. This is depicted as dots of the corresponding color. The reduced OLS error, given the signal pattern $I_f(\beta^0)$, can be computed by replacing the design matrix $X$ in (\ref{SLOPE-estimator_n}) with the reduced $XU_{\beta^0}$, where $U_{\beta^0}$ is a pattern matrix depending on $f$, as 
\begin{equation*}
    \hat{u}_{OLS(I_f(\beta^0))}\sim \mathcal{N}(0, \sigma^2 (U_{\beta^0}^TCU_{\beta^0})^{-1}).
\end{equation*}
In Figure \ref{fig: comparing RMSE}, the Lasso penalty is equal to $\alpha$, the SLOPE penalty $\alpha[1.6, 1.2, 0.8, 0.4]$, and the Fused Lasso penalty is $\alpha (\sum_{i=1}^3\vert\beta_{i+1}-\beta_i\vert+\sum_{i=1}^4\vert\beta_i\vert)$. The covariance $C$ is given by 
\begin{equation*}
    C=\left(\begin{matrix}
        1 & 0 & 0.8 & 0\\
        0 & 1 & 0 & 0.8\\
        0.8 & 0 & 1 & 0\\
        0 & 0.8 & 0 & 1
    \end{matrix}\right).
\end{equation*}
\begin{figure}[ht]
      \centering
	     \begin{subfigure}{0.5\linewidth}
		 \includegraphics[width=\linewidth]{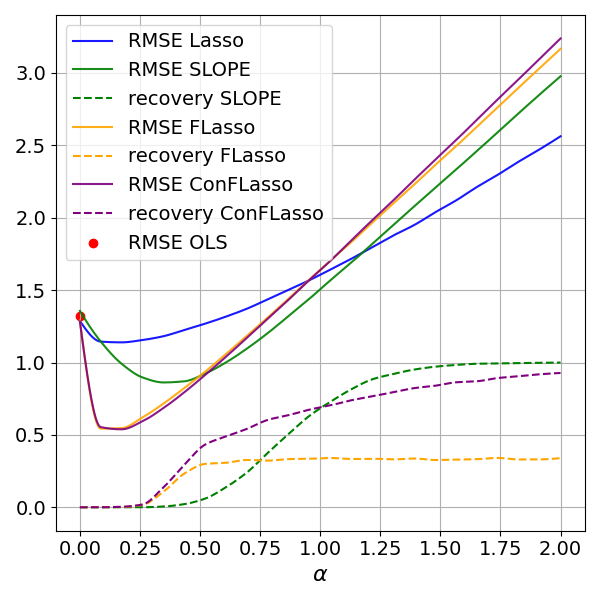}
		 \caption{  
       $\beta^0 = [0, 0, 0, 1, 1, 1, 3, 3, 3, 2, 2, 2]$}
	      \end{subfigure}
	\caption{Comparing root mean squared error (RMSE) for different methods together with the probability of pattern recovery, (i.e. correctly identifying all zeros and all clusters). }
	\label{fig:concavifie lasso plot}
\end{figure}
We also note that the choice for the SLOPE sequence is not optimal and can be improved by choosing a different tuning, depending on the signal. For example, in b), the penalty sequence $\alpha[4,0,0,0]$ achieves better estimation and pattern recovery than the linear OSCAR sequence above. 

In Figure \ref{fig:concavifie lasso plot}, the Lasso penalty is equal to $\alpha$, and the SLOPE penalty sequence is linear $\alpha \lambda_i$ with $\lambda_i=12i/\sum_{i=1}^{12} i$, so the total penalization is $\sum_{i=1}^{12}\lambda_i=12$. The Concavified Fused Lasso is set to $\alpha( \sum_{i=1}^8a_i\vert \beta_{i+1}-\beta_i\vert +  \sum_{i=1}^9 \vert\beta_i\vert)$, with a concave clustering sequence $a_i = \nu(1 + \kappa i(9-i))$ with concavity parameter $\kappa=0.04$ and clustering parameter $\nu=0.8$. 
The Fused Lasso has $\kappa=0$ and the clustering parameter is set to be the average $\nu=(1/8)\sum_{i=1}^8 a_i$ of the Concavified Fused Lasso. The covariance $C$ is block-diagonal consisting of four $3\times3$ unit diagonal blocks with $0.8$ off-diagonal entries; $\sigma = 0.2$ respectively. 
\subsection{three-step procedure}
\begin{figure}[bt]
      \centering
	     \begin{subfigure}{0.32\linewidth}
		 \includegraphics[width=\linewidth]{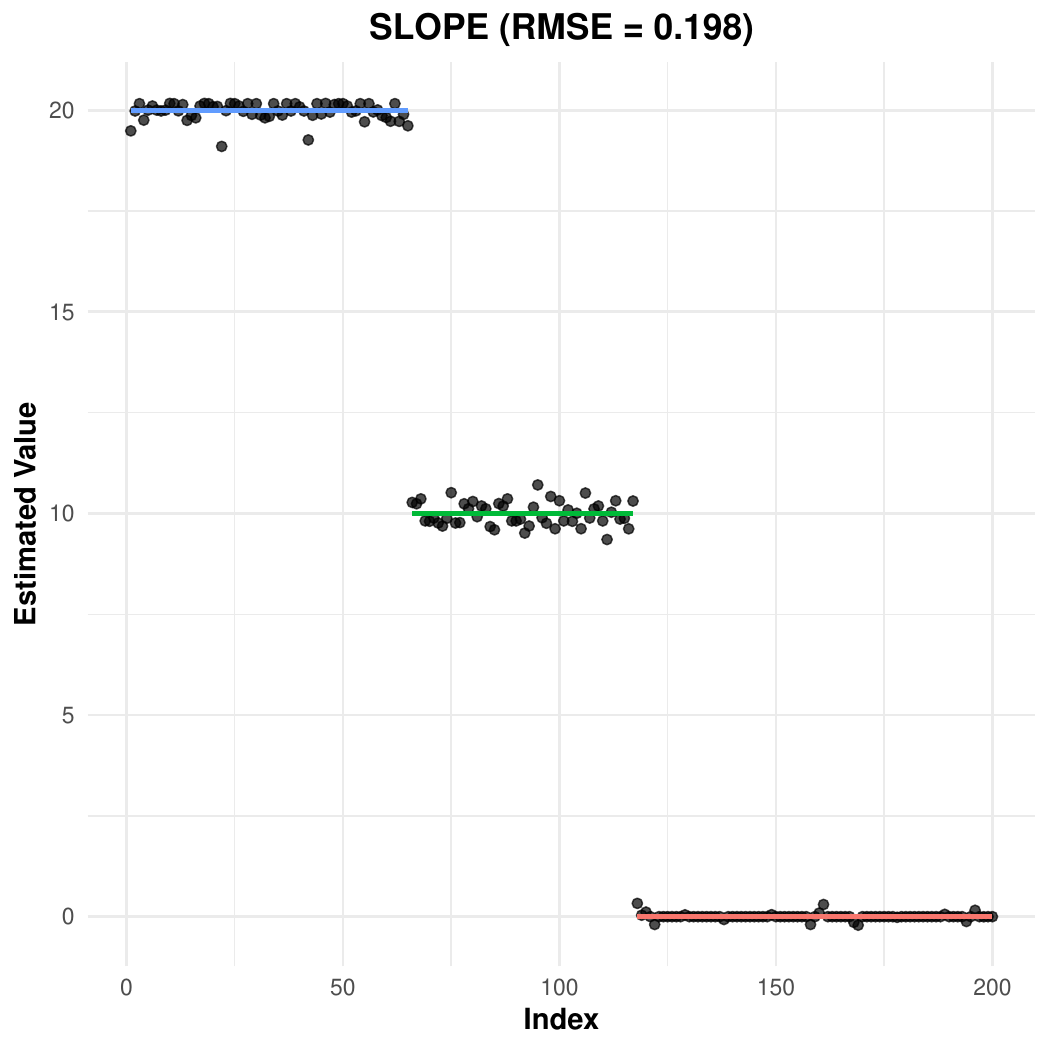}
	      \end{subfigure}
	      \begin{subfigure}{0.32\linewidth}
		  \includegraphics[width=\linewidth]{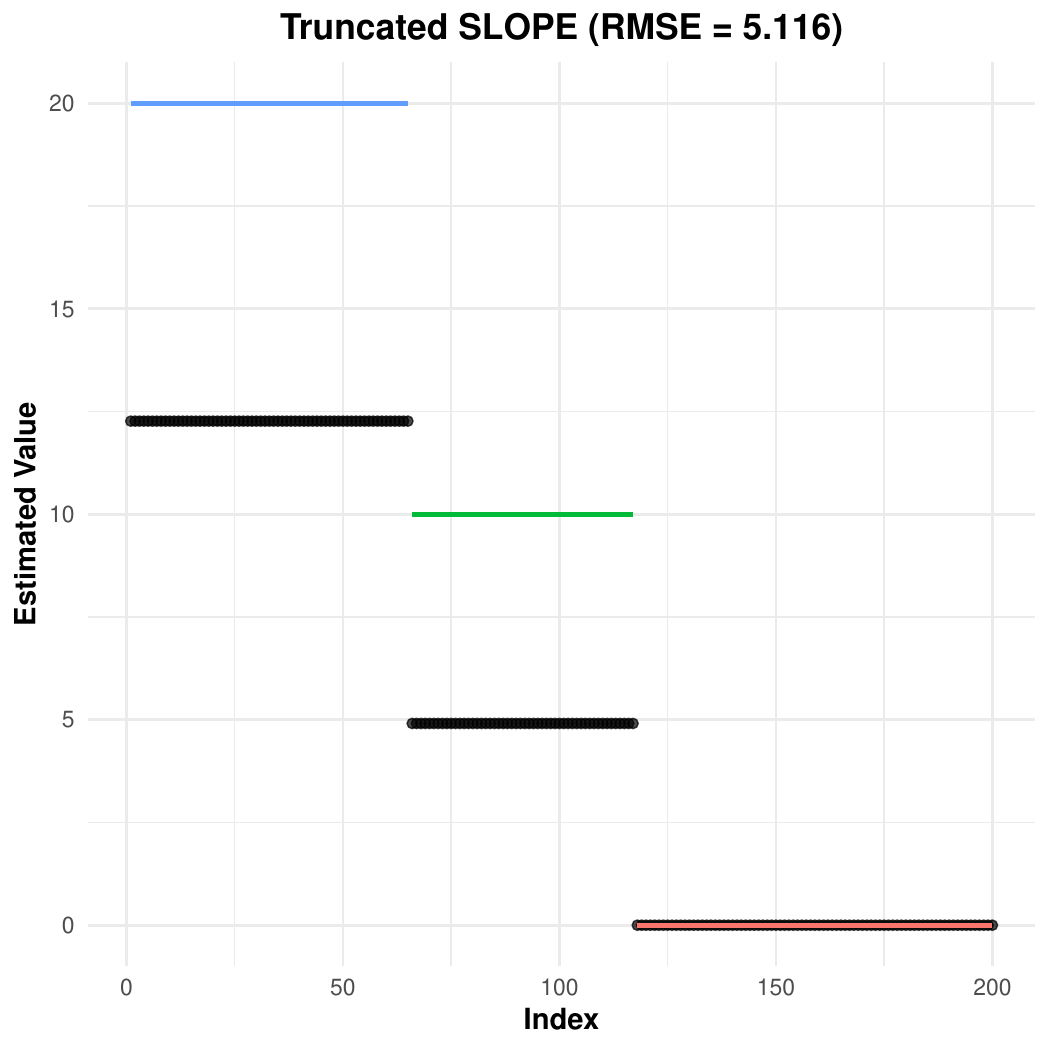}
	       \end{subfigure}
        \begin{subfigure}{0.32\linewidth}
		  \includegraphics[width=\linewidth]{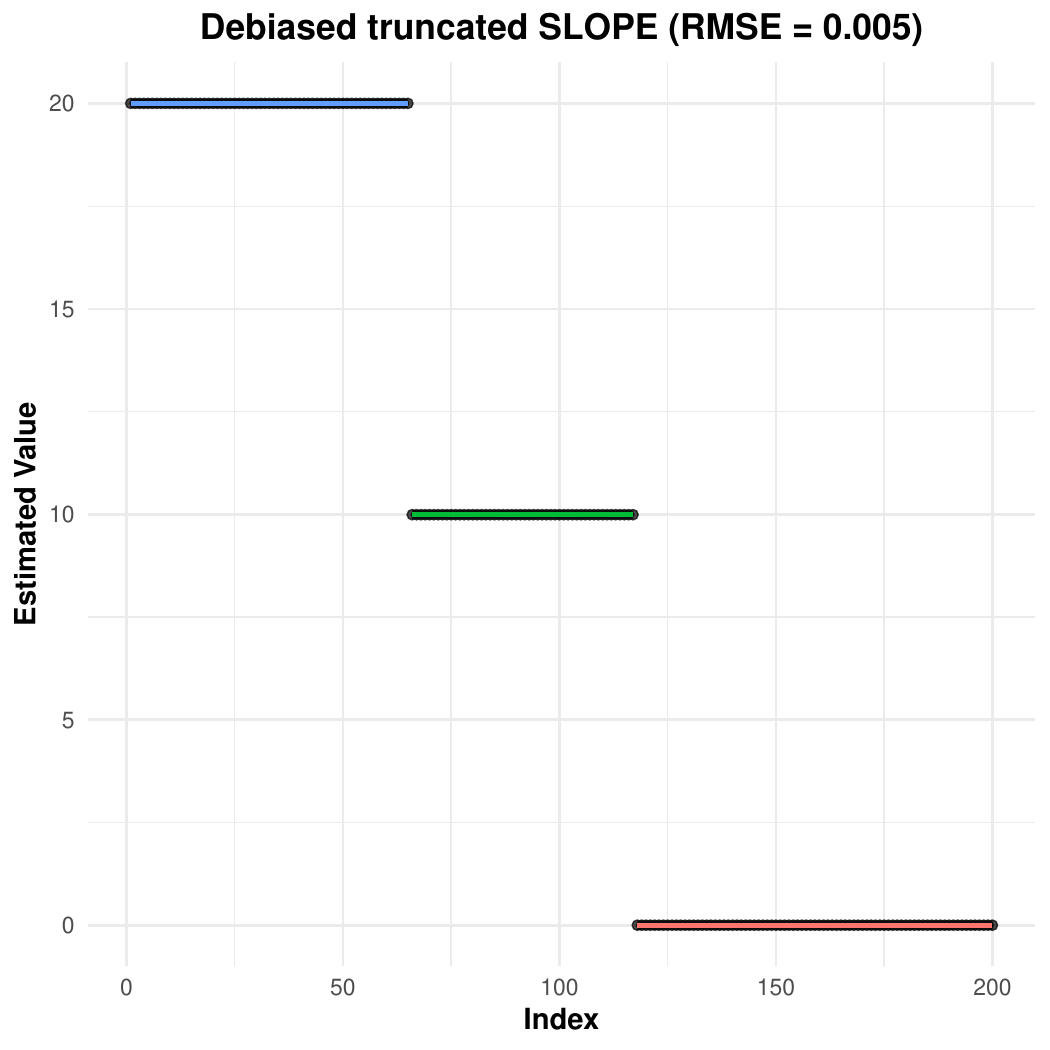}
	       \end{subfigure}
	\caption{The black dots shows the estimated $\beta$ coefficients for all three steps in the Three-step procedure for SLOPE. The lines corresponds to the true coeffients.}
	\label{fig: two-step procedure}
\end{figure}

To illustrate the three-step estimation procedure in a high-dimensional scenario, we simulate data as follows.\footnote{ The code can be found at https://github.com/IvanHejny/Three-step-procedure-for-SLOPE.git}
The design matrix $X$ is $n \times p$ with $n = 100$ and $p = 200$. Each row of $X$ is sampled i.i.d.\ from a $\mathcal{N}(0, C)$ distribution, where $C$ is a block-diagonal covariance matrix consisting of $20$ blocks. Each block is a $10 \times 10$ correlation matrix whose diagonal entries are $1$ and off-diagonal entries are $0.8$. The true coefficient vector $\beta^0$ has three clusters: $\beta^0_i = 20$ for $1 \le i \le 65$, $\beta^0_i = 10$ for $66 \le i \le 117$, and $\beta^0_i = 0$ for $118 \le i \le 200$. Finally, the noise term $\varepsilon$ is drawn from $\mathcal{N}(0,\sigma^2 I)$ with $\sigma = 0.8$.

\medskip
\noindent
Figure~\ref{fig: two-step procedure} illustrates the three-step estimation procedure~\eqref{two-step procedure} and~\eqref{three-step procedure}. 
\begin{itemize}
\item \textbf{Step 1 (left):} We obtain the initial SLOPE estimate $\hat{\beta}^{(1)}$, using the Benjamini--Hochberg sequence $\boldsymbol{\lambda} = 0.07\,n^{-1/2}\mathrm{BH}(0.5)$. 
\item \textbf{Step 2 (middle):} We form the truncated estimate~\eqref{two-step procedure} 
$
\hat{\beta}^{(2)} \;=\; \mathrm{Prox}_{42\,n^{-1/2}\,J_{\boldsymbol{\lambda}}(\cdot)}\bigl(\hat{\beta}^{(1)}\bigr).
$
\item \textbf{Step 3 (right):} We compute the reduced OLS estimate~\eqref{three-step procedure}.
\end{itemize}

\noindent
From the figure, we observe that:
\begin{enumerate}
\item In Step~1, the overall magnitude and support of the coefficients are identified reasonably well, but the cluster structure is not recovered.
\item In Step~2, the clusters are recovered, although this step introduces a heavy bias.
\item In Step~3, the reduced OLS step corrects this bias and yields more accurate coefficient estimates.
\end{enumerate}

\section{Discussion}
In this article, we proposed a general theoretical framework for the asymptotic analysis of pattern recovery for a broad class of regularizers, including Lasso, Fused Lasso, Elastic Net, or SLOPE. We argue that the ``classical'' asymptotic framework, where the model dimension $p$ is fixed and $n\rightarrow\infty$, can provide deep insight into both the model selection properties and the estimation accuracy. This is achieved by studying the asymptotic distribution of the error $\hat{u}_n=\sqrt{n}(\hat{\beta}_n-\beta^0)$. 
We showed that the analysis of pattern convergence for regularizers requires a separate treatment, as it is not a simple consequence of the distributional convergence of $\hat{u}_n$. We solved this by using the Hausdorff distance as a suitable mode of convergence for subdifferentials, which leads to the desired pattern recovery.

We demonstrated how our asymptotic analysis can lead to new methodological insights, such as concavifying the penalty coefficients in Fused Lasso; a remedy for its inability to recover its own model under the random design with independent regressors. We believe that our framework provides a fertile ground for further such discoveries.

We conducted a small simulation study to compare the performance of different regularizers in terms of their estimation accuracy and pattern recovery. We illustrate that performance depends on whether the estimator can ``access the underlying structure'' of the signal. We observed that SLOPE, with the strictly decreasing sequence of the tuning parameters, can take advantage of general non-consecutive cluster structures, which are invisible to Lasso or the Fused Lasso, and performs reasonably well for various scenarios. However, in cases where the clustering structure is absent and the signal is relatively sparse, Lasso (corresponding to the constant SLOPE sequence) can be more efficient in discovering the respective sparsity pattern. Similarly, when clustering occurs between prespecified ``neighboring'' regressors, then the specialized Fused Lasso can outperform both SLOPE and Lasso.

Furthermore, we proposed an easy yet effective two-step procedure that recovers the true model pattern for any covariance structure of the regressors, thus circumventing the rather restrictive irrepresentability condition. By employing this as a dimensionality reduction tool, we believe that there is great potential for further methodological development, especially in combination with third-order methods.

The asymptotic results presented in this paper focus on classical asymptotics, where the model dimension $p$ is fixed and $n$ diverges to infinity. Our analysis reveals that even in this classical setup, deriving results on pattern convergence requires the development of new tools and substantially more care compared to the convergence of the vector of parameter estimates. We believe that our framework, based on the weak convergence of patterns, can be extended to the analysis of regularizers in a high-dimensional setup. We consider our work an important first step in this direction.

\section{Acknowledgments}
The authors acknowledge the support of the Swedish Research Council, grant no. 2020-05081. We would especially like to thank Alexandre B. Simas for his revisions and insightful comments on the Hausdorff distance. We would also like to thank Elvezio Ronchetti for the discussions on the robust versions of SLOPE and further A.W. van der Vaart, Wojciech Reichel, Ulrike Schneider, Piotr Graczyk, Bartosz Ko{\l}odziejek, Tomasz Skalski, and Patrick Tardivel for their helpful comments and discussions.

\bibliographystyle{plain}
\newpage
\bibliography{citation}

\newpage

\begin{appendix}
\section{Appendix}\label{Appendix}
\subsection{directional derivative for SLOPE}\label{appendix directional derivative for SLOPE}
Here we compute the directional derivative for SLOPE $J'_{\boldsymbol{\lambda}}(x;u)$ at $x$ in direction $u$. For fixed $u\in\mathbb{R}^p$ there exists a permutation $\pi$, which sorts $\vert x+\varepsilon u\vert$ for all sufficiently small $\varepsilon$, i.e. $\vert x+\varepsilon u\vert_{\pi^{-1}(1)}\geq...\geq\vert x+\varepsilon u\vert_{\pi^{-1}(p)}$ as $\varepsilon\downarrow 0$. At the same time we have $\vert x\vert_{\pi^{-1}(1)}\geq...\geq\vert x\vert_{\pi^{-1}(p)}$. Consequently, for such $\pi$ and $\varepsilon>0$ sufficiently small;
\begin{align*}
    J_{\boldsymbol{\lambda}}(x+\varepsilon u)-J_{\boldsymbol{\lambda}}(x)
    &= \sum\limits_{j=1}^p\boldsymbol{\lambda}_j\left[\vert x+\varepsilon u\vert_{\pi^{-1}(j)}-\vert x\vert_{\pi^{-1}(j)}\right]\nonumber\\
    & = \sum\limits_{i=1}^p\boldsymbol{\lambda}_{\pi(i)}\left[\vert x_i+\varepsilon u_i\vert-\vert x_i\vert\right]\nonumber\\    &=\sum\limits_{i=1}^p\boldsymbol{\lambda}_{\pi(i)}\left[\varepsilon u_i sgn(x_i) \mathbb{I}[x_i\neq 0]+ \varepsilon\vert u_i\vert\mathbb{I}[x_i=0]\right]\nonumber.
\end{align*}
Therefore 
\begin{align*}
    J'_{\boldsymbol{\lambda}}(x;u) = \sum\limits_{i=1}^p\boldsymbol{\lambda}_{\pi(i)}\left[ u_i sgn(x_i) \mathbb{I}[x_i\neq 0]+ \vert u_i\vert\mathbb{I}[x_i=0]\right].
\end{align*}
\subsection{Subdifferential for SLOPE}\label{SLOPE subdifferential Appendix}
Let $\mathcal{S}$ denote the set of all signed permutations. Then
\begin{align*}
    J_{\boldsymbol{\lambda}}(x)&=\max\{\langle P\boldsymbol{\lambda},x\rangle: P\in \mathcal{S}\},\\
    \partial J_{\boldsymbol{\lambda}}(x)&= con\{P\boldsymbol{\lambda}: P\in I(x)\},\\
    I(x)&=\operatorname{argmax}_{P\in\mathcal{S}}\langle P\boldsymbol{\lambda},x\rangle,
\end{align*}
 More explicitly, let $\mathcal{I}(x)=\{I_0,I_1,\dots,I_m\}$ be the partition of $\{1,\dots,p\}$ into the clusters of $x$. Let $S_x$ be the diagonal matrix, s.t. $(S_x)_{ii}=1$ for $i\in I_0$, and $(S_x)_{ii}=sgn(x_i)$ else. Also, fix $\Pi_x \in \mathcal{S}$, such that 
\begin{equation*}
    \langle \Pi_x\boldsymbol{\lambda},\vert x\vert\rangle= J_{\boldsymbol{\lambda}}(x),
\end{equation*}
i.e. the maximum is attained. 
Finally, consider the group of symmetries of $\vert x\vert$ in $\mathcal{S}$:
\begin{align*}
    Sym(\vert x \vert)&=\{\Sigma \in \mathcal{S}: \Sigma \vert x\vert = \vert x\vert\},\\
&=\mathcal{S}^{+/-}_{I_0}\oplus\mathcal{S}_{I_1}\oplus...\oplus\mathcal{S}_{I_m}.
\end{align*}
For any $\Sigma\in Sym(\vert x\vert)$, also $\Sigma^T=\Sigma^{-1}\in Sym(\vert x\vert)$, and:
\begin{equation*}
    J_{\boldsymbol{\lambda}}(x)=\langle \Pi_x\boldsymbol{\lambda},\Sigma^T\vert x\vert\rangle=\langle S_x\Sigma\Pi_x\boldsymbol{\lambda},x\rangle.
\end{equation*}
Hence $I(x)=\{S_x\Sigma\Pi_x:\Sigma\in Sym(\vert x\vert)\}$, and
\begin{align*}
    \partial J_{\boldsymbol{\lambda}}(x)&=con\{S_x\Sigma\Pi_x\boldsymbol{\lambda}:\Sigma\in Sym(\vert x\vert)\}\\
    & =con\{\mathcal{S}^{+/-}_{I_0}\Pi_x\boldsymbol{\lambda}\}\oplus con\{S_x\mathcal{S}_{I_1}\Pi_x\boldsymbol{\lambda}\}\oplus..\oplus con\{ S_x\mathcal{S}_{I_m}\Pi_x\boldsymbol{\lambda}\}.
\end{align*}

For illustration, let $x=(\textcolor{green}{0}, \textcolor{blue}{2}, \textcolor{blue}{-2}, \textcolor{red}{1}, \textcolor{blue}{2}, \textcolor{red}{1} )^T$, $\mathcal{I}(x)=\{\textcolor{green}{\{1\}},\textcolor{red}{\{4,6\}},\textcolor{blue}{\{2,3,5\}}\}$. Then
\\
\def\re{\color{red}}
\def\bl{\color{blue}}
\def\gr{\color{green}}
\def\black{\color{black}}
\[ \Sigma = 
\left(\begin{matrix}
\gr-1 & 0 & 0 & 0 & 0 & 0 \\
0 & \bl0 & \bl1 & 0 & \bl0 & 0 \\
0 & \bl0 & \bl0 & 0 & \bl1 & 0\\
0 & 0 & 0 & \re0 & 0 & \re1\\
0 & \bl1 & \bl0 & 0 & \bl0 & 0\\
0 & 0 & 0 & \re1 & 0 & \re0
\end{matrix} \right)\in Sym(\vert x\vert)=\textcolor{green}{\mathcal{S}^{+/-}_{I_0}}\oplus\textcolor{red}{{S}_{I_1}}\oplus\textcolor{blue}{\mathcal{S}_{I_2}}
\qquad \Pi_x = \left(\begin{matrix}
0 & 0 & 0 & 0 & 0 & \gr1 \\
\bl1 & 0 & 0 & 0 & 0 & 0 \\
0 & \bl1 & 0 & 0 & 0 & 0\\
0 & 0 & 0 & \re1 & 0 & 0\\
0 & 0 & \bl1 & 0 & 0 & 0\\
0 & 0 & 0 & 0 & \re1 & 0
\end{matrix} \right)
\]
$\Sigma\vert x\vert=(\textcolor{green}{0}, \textcolor{blue}{2}, \textcolor{blue}{2}, \textcolor{red}{1}, \textcolor{blue}{2}, \textcolor{red}{1} )^T=\vert x\vert$\hspace{6.3cm}$\Pi_x\boldsymbol{\lambda}=( \textcolor{green}{\lambda_6}, \textcolor{blue}{\lambda_1}, \textcolor{blue}{\lambda_2},  \textcolor{red}{\lambda_4}, \textcolor{blue}{\lambda_3}, \textcolor{red}{\lambda_5})^T$

\subsection{Pattern space}\label{Pattern space Appendix}
We show (\ref{pattern space representation}), 
i.e. if $f$ satisfies (\ref{general linear framework}), then the following vector spaces are the same:
  \begin{align}
        i) &\hspace{0.3cm}span\{u: I(u)=I(x)\},\nonumber\\
        ii) &\hspace{0.3cm}par(\partial f(x))^{\perp},\nonumber\\
        iii) &\hspace{0.3cm}\{u\in\mathbb{R}^p: I_x(u)=I(x)\}\nonumber.
    \end{align}
\begin{proof}
Recall that $I_x(u)=\operatorname{argmax}_{i\in I(x)}\langle v_i, u\rangle$. Thus
\begin{align*}
    I_x(u)=I(x)&\iff \langle v_i,u\rangle= \langle v_j,u\rangle \;\;\forall i,j\in I(x)\\
    &\iff \langle w-\tilde{w}, u\rangle = 0 \;\;\;\forall w,\tilde{w}\in \partial f(x)\\
    &\iff u\in par(\partial f(x))^{\perp},
\end{align*}
hence $ii)=iii)$. Also, if $I(u)=I(x)$, then $I_x(u)=\operatorname{argmax}_{i\in I(u)}\langle v_i, u\rangle = I(u)=I(x)$, thus $i)\subset iii)$, because $iii)$ is a vector space. For the opposite inclusion, let $I_x(u)=I(x)$. Since, $I_x(u)=\operatorname{lim}_{\varepsilon\downarrow 0}I(x+\varepsilon u)$, we know that $I_x(u)=I(x+\varepsilon u)$ for every $\varepsilon>0$ small enough. Therefore $u= \varepsilon^{-1}((x+\varepsilon u) - x) \in span\{ u: I(u)=I(x)\}$, because $I(x+\varepsilon u)=I_x(u)=I(x)$. 
\end{proof}

Further, we show (\ref{subdifferential as a face of the dual ball}), that $\partial f(x)= \partial f(0) \cap (v_0+\langle U_x\rangle^{\perp})$, where $v_0\in\partial f(x)$, provided $g(x)=0$ in (\ref{general linear framework}).
\begin{proof}
Since $\partial f(x)\subset \partial f(0)$, and $\partial f(x)-v_0 \in par(\partial f(x))=\langle U_x\rangle^{\perp}$ by (\ref{pattern space representation})ii), the $\subset$ inclusion follows. For the opposite inclusion, 
let $v\in \partial f(0) \cap (v_0+\langle U_x\rangle^{\perp})$, we have $v=\sum_{i\in\mathcal{S}}\lambda_i v_i$, $\sum_{i\in\mathcal{S}}\lambda_i=1,\lambda_i\geq0,$ and at the same time $v=v_0+\sum_{i\in I(x)}\alpha_i (v_i-v_0)$, $\alpha_i\in \mathbb{R}$. We obtain 
\begin{align*}
    \sum_{i\in I(x)}(\lambda_i-\alpha_i)(v_i-v_0) + \sum_{i\notin I(x)}\lambda_i(v_i-v_0)=0.
\end{align*}
Since $I(x)=\operatorname{argmax}_{i\in\mathcal{S}}\langle v_i,x\rangle$, we get $\langle v_i-v_0, x\rangle = 0 \;\forall i\in I(x)$ and $\langle v_i-v_0, x\rangle < 0 \;\forall i\notin I(x)$. Taking the inner product of the above expression with $x$ gives $\sum_{i\notin I(x)}\lambda_i\langle v_i-v_0,x\rangle=0$. Consequently, $\lambda_i=0$ for all $i\notin I(x)$, and $v=\sum_{i\in I(x)}\lambda_i v_i\in con\{v_i: i\in I(x)\} = \partial f(x)$.  
\end{proof}
\subsection{limiting pattern}\label{appendix limiting pattern}
We prove that the limiting pattern $I_x(u):=\operatorname{argmax}_{i\in I(x)}\langle v_i,u\rangle$ equals $\lim_{\varepsilon\downarrow0}I(x+\varepsilon u)$, where $I(x)=\operatorname{argmax}_{i\in\mathcal{S}}\langle v_i,x\rangle$ and we recall the penalty is given by $f(x)=\max\{v_i^Tx: i\in\mathcal{S}\}$.
\begin{proof}
For any fixed $x,u\in\mathbb{R}^p$; $I(x+\varepsilon u)\subset I(x)$ eventually as $\varepsilon\downarrow 0$. Indeed, by contradiction, assume that $i\in I(x+\varepsilon u)$, but $i\notin I(x)$. Then $\langle v_i,x\rangle<\langle v_{i_0}, x\rangle$ for some $i_0\in I(x)$, and as a result for sufficiently small $\varepsilon>0$, $\langle v_i, x+\varepsilon u\rangle<\langle v_{i_0}, x+\varepsilon u\rangle$. It follows that $i\notin I(x+\varepsilon u)$, a contradiction. As a result for $\varepsilon\downarrow 0$ the pattern eventually stabilizes at
\begin{align}\label{stabilization of limiting pattern}
    I(x+\varepsilon u) & = \underset{i\in\mathcal{S}}{\operatorname{argmax}}\langle v_i,x +\varepsilon u\rangle\nonumber\\
    & = \underset{i\in I(x)}{\operatorname{argmax}}\langle v_i,x +\varepsilon u\rangle\nonumber\\
    & = \underset{i\in I(x)}{\operatorname{argmax}}\langle v_i,u\rangle = I_x(u),
\end{align}
where we have used that $\langle v_i,x \rangle$ is the same for every $i\in I(x)$ by the definition of pattern $I(x)=\operatorname{argmax}_{i\in\mathcal{S}}\langle v_i,x\rangle$, which proves the claim.
\end{proof}

\subsection{Failure of weak pattern convergence}\label{appendix counterexample pattern convergence}
We present an example of a convex penalty, for which the error $\hat{u}_n$ converges in distribution to $\hat{u}$, but $sgn(\hat{u}_n)$ does not converge to $sgn(\hat{u})$. Consider the penalty $f(x)=\max\{x_1^2,x_2\}$ on $\mathbb{R}^2$, and let $f_n = n^{1/2}f$. Figure \ref{fig: flattening parabola}. illustrates, why the $sgn(\hat{u}_n)$ fails to converge to $sgn(\hat{u})$. Formally, for $C_n$ and $W_n$ as in (\ref{W_n C_n convergence}), by Theorem \ref{sqrt-asymptotic}, 
\begin{align*}
    \hat{u}_n=\sqrt{n}(\hat{\beta}_n-\beta^0)=&\operatorname{argmin}u^{T}C_n u/2-u^{T}W_n+ n^{1/2}[f(\beta^0+u/\sqrt{n})-f(\beta^0)]\\
    \overset{d}{\longrightarrow}&\operatorname{argmin}{u^{T}Cu/2-u^{T}W+ f'({\beta^0};u)}=:\hat{u}.
\end{align*} 
For $\beta^0=0$, we have $n^{1/2}[f(\beta^0+u/\sqrt{n})-f(\beta^0)]=\max\{n^{-1/2}u_1^2,u_2\}=:g_n(u)$, and on the half line $\mathcal{K}=\{u_1>0, u_2=0\}$; the subdifferential $\partial g_n(u)=(2u_1/\sqrt{n},0)^T$ is zero-dimensional. We obtain
\begin{equation*}
    \mathbb{P}\left[\hat{u}_{n}\in \mathcal{K}\right]=\mathbb{P}\left[W_n\in \left\{C_n \begin{pmatrix} u_1 \\ 0 \end{pmatrix}+\begin{pmatrix} 2u_1/\sqrt{n} \\ 0 \end{pmatrix}:u_1>0\right\}\right]=0\hspace{0,2cm}\forall n,
\end{equation*} 
provided $W_n$ is absolutely continuous w.r.t. the Lebesgue measure.
Furthermore, from (\ref{directional derivative equation}) we get $f'(0;u)=\max\{\langle(0, 0),(u_1,u_2)\rangle, \langle(0,1),(u_1,u_2)\rangle\}=\max\{0, u_2\}$, hence on $\mathcal{K}$ the subdifferential $\partial f'(0;u)=con\{(0,0)^T,(0,1)^T\}$ is one-dimensional. We get
\begin{equation*}
    \mathbb{P}\left[\hat{u}\in \mathcal{K}\right]=\mathbb{P}\left[W\in \left\{C \begin{pmatrix} u_1 \\ 0 \end{pmatrix}+\partial f'(0;u):u_1>0\right\}\right]>0,
\end{equation*}
since $C_{11}>0$. In particular, $sgn(\hat{u}_n)$ does not converge weakly to $sgn(\hat{u})$, despite the weak convergence of $\hat{u}_n$ to $\hat{u}$.

Observe that $\hat{u}_n$ puts positive mass on the parabola $\{u_2=n^{-1/2}u_1^2\}$, where $\partial g_n(u)$ is one-dimensional, whereas $\hat{u}$ puts positive mass on the tangential space of the parabola at $0$ given by $\{u_2=0\}$.
\begin{figure}[ht]
    \centering
\begin{tikzpicture}
  \draw[->] (0, 0) -- (0, 1) node[above] {};
  \draw[scale=0.5, domain=-4:4, line width = 0.3mm, dash pattern = on 4pt off 1pt, variable=\x, blue] plot ({\x}, {\x*\x/13}) node[above]{$\hat{u}_n$};
  \draw[scale=0.5, domain=-4.5:4.5, line width = 0.3mm, dash pattern = on 4pt off 1pt, variable=\x, blue] plot ({\x}, {\x*\x/21}) node[right]{$\hat{u}_{n+1}$};
  \draw[scale=0.5, domain=-5:5, line width = 0.3mm, dash pattern = on 2pt off 1pt, variable=\x, purple] plot ({\x}, {0}) node[right]{$\hat{u}$};
\end{tikzpicture}
\caption{$\hat{u}_n$ puts mass on parabola}
\label{fig: flattening parabola}
\end{figure}

More precisely, the Lebesgue decompositions of $\hat{u}_n$ and $\hat{u}$ w.r.t. the Lebesgue measure yield different singular sets; the parabola and the x-axis respectively. This gives some intuition for why linearity of the functions in the penalty $f=\max\{f_1,..,f_N\}$ is essential for convergence on convex sets. 

Notice that if we allow the pattern to change with $n$, we get weak convergence of $I_n(\hat{u}_n)$ to $I(\hat{u})$, which can be argued by the Portmanteau Lemma. Here, the pattern $I_n$ can be identified with the three regions of $\mathbb{R}^2$ determined by the parabola $\{u_2=n^{-1/2}u_1^2\}$.  

\subsection{Proofs}\label{apppendix proofs}
\begin{proof} (Lemma \ref{Hausdorff lemma})
 Let $\delta>0$ be arbitrary, and let $B^{-\delta}:=\{x\in B: d(x,B^c)>\delta\}$ denote the open $\delta-$ interior of $B$, where $B^c:=\mathbb{R}^p\setminus B$ is the complement of $B$. Note that $B^{-\delta}$ is open by continuity of $x\mapsto d(x, B^c)$. Also we denote the interior of a set $B$ by $B^{\circ}$. Since $B_n\overset{d_H}{\longrightarrow}B$, it follows that $B_n\subset B^{\delta}$ eventually. Similarly, for all sufficiently large $n$, we have $B\subset B_n^{\delta}$, thus $B^{-\delta}\subset (B_n^{\delta})^{-\delta}=B_n^{\circ}\subset B_n$, where the equality follows from convexity \footnote{Convexity is necessary: The annuli $B_n=\overline{B_1(0)}\setminus B_{1/n}(0)\overset{d_H}{\longrightarrow} \overline{B_1(0)}$, but $\overline{B_1(0)}^{-\delta}\not\subset B_n$. } and the fact that $(B_n^{\delta})^{-\delta}$ is an open set. As a result, for any $\delta>0$; $ B^{-\delta}\subset B_n\subset B^{\delta}\hspace{0,2cm} \text{eventually}$. Moreover, since $B$ is convex, and $W$ is absolutely continuous w.r.t. the Lebesgue measure, one can show that for every $\varepsilon>0$ there exists \footnote{In fact, the bounds with tubular sets hold uniformly over all convex sets; i.e., for each $\varepsilon>0$ there even exists a $\delta>0$ such that (\ref{B^delta approx}) holds for every convex set $B$. 
} a $\delta>0$ such that 
\begin{equation}\label{B^delta approx}
    \mathbb{P}[W\in B^{\delta}\hspace{0,1cm}]-\varepsilon\leq\mathbb{P}[W\in B]\leq\mathbb{P}[W\in B^{-\delta}]+\varepsilon,
\end{equation}
for an analogous statement, see for example proof of Corollary 2.7.9 \cite{vanderVaart1996}.  Consequently, for any $\varepsilon>0$ we can choose $\delta>0$ sufficiently small such that:
\begin{align*}
    \limsup_{n\rightarrow\infty}\mathbb{P}[W_n\in B_n]&\leq \limsup_{n\rightarrow\infty} \mathbb{P}[W_n\in B^{\delta}]\leq\mathbb{P}[W\in B^{\delta}]\leq\mathbb{P}[W\in B]+\varepsilon
    \\[5pt]
    \liminf_{n\rightarrow\infty}\mathbb{P}[W_n\in B_n]&\geq \liminf_{n\rightarrow\infty} \mathbb{P}[W_n\in B^{-\delta}]\geq\mathbb{P}[W\in B^{-\delta}]\geq\mathbb{P}[W\in B]-\varepsilon,
\end{align*}
where we have used the Portmanteau Lemma and the fact that $B^{\delta}$ and $B^{-\delta}$ are closed and open respectively. 
This shows the desired convergence $\mathbb{P}[W_n\in B_n]\longrightarrow \mathbb{P}[W\in B]$.
\end{proof}

\begin{proof}(Proposition \ref{concavification of Fused Lasso})
Recall the Fused Lasso penalty: $$f_A(\beta)=\lambda\Vert A\beta\Vert_1 = \lambda\sum_{i=1}^{p-1} a_i\vert \beta_{i+1}-\beta_i\vert + \lambda\sum_{i=1}^p a\vert \beta_i\vert, $$
with $a_i>0\;\forall i$ and $a, \lambda>0$. To recover all patterns, it is both sufficient and necessary that for every $\beta^0\in\mathbb{R}^p$;
$0$ lies in the relative interior of $(I-P)\partial f_A(\beta^0)$. We decompose this condition into a more tangible form. First, note that $\partial f_A(\beta^0) = A^T \partial f(A\beta^0)$. 
Let $\mathcal{I}(\beta^0)=\{I_1,I_2,\dots I_{m-1},I_m\}$ be the partition of $\beta^0$ into consecutive clusters. (Here, a cluster is a consecutive set of indices where $\beta^0_i$ have the same values.)

First, we assume that $a=0$. The pattern space of Fused Lasso is $\langle U_{\beta^0}\rangle=span\{\mathbf{1}_{I}: I\in \mathcal{I}(\beta^0)\}$, which is the span of $U_{\beta^0}=(\mathbf{1}_{I_1},\dots,\mathbf{1}_{I_m})$. Let $P=U_{\beta^0}(U_{\beta^0}^TU_{\beta^0})^{-1}U_{\beta^0}$ be the projection onto $\langle U_{\beta^0}\rangle$. The projection averages the values on each cluster, and decomposes as a block-diagonal matrix $P=P_{I_1}\oplus\dots \oplus P_{I_m}$, with $P_I= \mathbf{1}_{I} \mathbf{1}_{I}^T/\vert I\vert $. 
Given an arbitrary invertible matrix $E$, the irrepresentability condition is equivalent to $0\in  E(I-P)A^T ri(\partial f(A\beta^0))$. Here, we let $E=E_{I_1}\oplus\dots\oplus E_{I_m}$, with $(E_I)_{ij}=1$ if $i\leq j$ and $i,j\in I$,  $(E_I)_{ij}=0$ else. Now $E(I-P)=E_{I_1} (I-P_{I_1})\oplus\dots\oplus E_{I_m}(I-P_{I_m})$, and it suffices to verify if $0\in E_I(I-P_I)A^T\partial f(A\beta^0)$ for all $I\in\mathcal{I}(\beta^0)$.
For a cluster $I$ of size $k$, one can check that:
\begin{align*}
E_I(I-P_I)\cong\frac{1}{k}
\begin{pmatrix}
  k-1 & -1 & -1 & \hdots & -1 \\
  k-2& k-2 & -2 &\hdots & -2 \\
k-3& k-3 & k-3 &\hdots & -3 \\
  \vdots & \vdots & \vdots & \ddots & \vdots \\
  1 & 1 & 1 & \hdots & -(k-1)\\
  0 & 0 & 0 & \hdots & 0
\end{pmatrix},
\end{align*}
We shall call an inner cluster $I_j$ \textit{monotone}, if $(\beta^0_{I_{j-1}},\beta^0_{I_{j}},\beta^0_{I_{j+1}})$ is monotone, otherwise, we call $I_j$ \textit{extremal}.
Let $I\in\{I_2,\dots, I_{m-1}\}$ be an inner cluster. Denoting the corresponding clustering penalties $(a_0^I, a_1^I,\dots,a_{k-1}^I,a_k^I)$, one can verify:
\begin{align*}
E_I(I-P_I)A^T\partial f(A\beta^0)\cong
\underbrace{\frac{1}{k}
\begin{pmatrix}
  -(k-1) & k & 0 & 0 & \hdots & 0 & -1 \\
  -(k-2) & 0 & k & 0 &\hdots & 0 & -2 \\
  -(k-3) & 0 & 0 & k &\hdots & 0 & -3 \\
  \vdots & \vdots & \vdots & \vdots & \ddots & \vdots \\
  -1 & 0 & 0 & 0 & \hdots & k & -(k-1)\\
   0 & 0 & 0 & 0 & \hdots & 0 & 0
\end{pmatrix}
}_{\cong \; E_I(I-P_I)A^T}
\underbrace{
\begin{pmatrix}
s_1a^I_0\\
[-a^I_1,a^I_1]\\
\vdots\\
[-a^I_{k-1},a^I_{k-1}]\\
s_2a^I_k
\end{pmatrix},
}_{\cong\; \partial f(A\beta^0)}
\end{align*}
where $s_1,s_2\in\{1,-1\}$, with $s_1=s_2$ if $I$ is monotone and $s_1\neq s_2$ if $I$ is extremal. 
Crucially, zero will fall into the interior, if and only if
\begin{equation}\label{fused lasso irrep equation}
((k-i)/k) s_1a^I_0 + (i/k) s_2a^I_k \in (-a^I_i,a^I_i)
\end{equation}

for every $1\leq i \leq k -1$. This is satisfied if $(a_0^I,\dots ,a_k^I)$ is strictly concave. For a boundary cluster $I=I_1$ resp. $I=I_m$, the above condition remains the same, but with setting $a^{I_1}_0=0$ resp. $a^{I_m}_{k_m}=0$. 
Then concavity of $(0, a^{I_1}_1,\dots,a^{I_1}_{k_1})$ resp. $(a^{I_m}_0,\dots,a^{I_m}_{k-1}, 0)$, yields the above condition, (irrespective of $s_1,s_2$). This shows that strict concavity of $(0,a_1,\dots,a_{p-1},0)$ is sufficient for recovering all non-zero clusters.

Conversely, strict concavity is also necessary. A penalty sequence, which is not strictly concave, contains a triple $(a_{i_1}, a_{i_2},a_{i_3})$, $0\leq i_1<i_2<i_3\leq p$, with $((i_3-i_2)/(i_3-i_1))a_{i_1}+((i_2-i_1)/(i_3-i_1))a_{i_3}\geq a_{i_2}$. Seting $I=\{i_1+1,\dots, i_3\}$, $k=\vert I\vert = i_3-i_1$ and $a_0^I=a_{i_1}, a_i^I=a_{i_2}, a_k^I=a_{i_3}$, $i=i_2-i_1$, this implies the converse of (\ref{fused lasso irrep equation}), i.e:
\begin{equation*}
    ((k-i)/k)s_1 a^I_0 + (i/k) s_2a^I_k \notin (-a^I_i,a^I_i),
\end{equation*}
whenever $s_1=s_2$ or $a_0^I=0$ or $a_k^I=0$. If $0<i_1$ and $i_3<p$, for a monotone $I$, $s_1=s_2$. If $i_1=0$ or $i_3=p$, we get $a^I_0=0$ resp. $a^I_k=0$. In either case, (\ref{fused lasso irrep equation}) is violated, thus the cluster $I$ cannot be recovered with high probability.  

Now, assume sparsity penalty $a>0$ in $A$, and let $I\in\mathcal{I}(\beta^0)$ be a zero cluster (of consecutive zeros). Then $P_I=0$, and one can verify
\begin{align*}
E_I(I-P_I)A^T\partial f(A\beta^0)\cong
\begin{pmatrix}
-a^I_0&+&[-a^I_1,a^I_1]&+&1[-a,a]\\
-a^I_0&+&[-a^I_2,a^I_2]&+&2[-a,a]\\
\vdots\\
-a^I_0&+&[-a^I_{k-1},a^I_{k-1}]&+&(k-1)[-a,a]\\
-a^I_0&+&-a^I_k&+&k[-a,a]
\end{pmatrix}.
\end{align*}
This set will contain $0$ in its interior, provided that $a>\max\{a_i +a_{i+1} : 0\leq i\leq p-1\}$. The condition is easily satisfied for the first $k-1$ equations, with much room to spare. However, it is also necessary in case $\vert I\vert=k=1$, where the last row yields $-a_0^I-a_1^I + [-a,a]$. Then $0\in ri(-a_0^I-a_1^I + [-a,a])$ if and only if $a>a_0^I+a_1^I$. This finishes the proof.
\end{proof}

\begin{proof}(Proposition \ref{attainability proposition})
By the optimality (\ref{main optimality condition}),
\begin{align*}
    I(\hat{u})=\mathfrak{p} &\iff W \in C I^{-1}(\mathfrak{p}) + \partial f(\mathfrak{q})\\
    &\iff C^{-1/2}(W-v_0)\in C^{1/2}I^{-1}(\mathfrak{p}) + C^{-1/2}(\partial f(\mathfrak{q})-v_0),
\end{align*}
for any $v_0\in\partial f(\mathfrak{q})$.
This event occurs with positive probability if and only if the above sum is a full-dimensional\footnote{We define the dimension of a set as the dimension of its affine space.} subset in $\mathbb{R}^p$, because $C$ is invertible and $W$ is continuous w.r.t. the Lebesgue measure.
 By (\ref{general orthogonality in subdifferentials}), $C^{1/2}I^{-1}(\mathfrak{p})\perp C^{-1/2}(\partial f(\mathfrak{q})-v_0)$, because $\langle U_{\mathfrak{p}}\rangle = span\{I^{-1}(\mathfrak{p})\}$ and $\partial f(\mathfrak{q})\subset \partial f(\mathfrak{p})$. Therefore, 
 \begin{align*}
     dim(C^{1/2}I^{-1}(\mathfrak{p}) + C^{-1/2}(\partial f(\mathfrak{q})-v_0))=dim(C^{1/2}I^{-1}(\mathfrak{p})) + dim(C^{-1/2}(\partial f(\mathfrak{q})-v_0)),
 \end{align*}
 which equals $p$ if and only if $dim(\partial f(\mathfrak{q}))=dim(\partial f(\mathfrak{p}))$. By (\ref{pattern space representation}), this is equivalent to $dim\langle U_{\mathfrak{q}}\rangle=dim\langle U_{\mathfrak{p}}\rangle$, which is in turn equivalent to  $\langle U_{\mathfrak{q}}\rangle=\langle U_{\mathfrak{p}}\rangle$, since $\langle U_{\mathfrak{p}}\rangle\subset\langle U_{\mathfrak{q}}\rangle$.
\end{proof}


\subsection{Proximal operator}\label{proximal operator}
If the proximal operator to $u\mapsto f'(\beta^0;u)$ is known, one can solve $(\ref{V(u)})$ using proximal methods. Here we compute the proximal operator for the directional SLOPE derivative $u\mapsto J'_{\boldsymbol{\lambda}}({\beta^0};u)$:
\begin{equation*}
\text{prox}_{J'_{\boldsymbol{\lambda}}(\beta^0,\cdot)}(y):=\underset{u\in\mathbb{R}^p}{\operatorname{argmin }}\hspace{0.2cm} (1/2)\Vert u-y\Vert_2^2 + J'_{\boldsymbol{\lambda}}({\beta^0};u)
\end{equation*}


Let $\mathcal{I}(\beta^0)=\{I_0, I_1, ..,I_m\}$ be the partition of $\beta^0$ into the clusters of the same magnitude. The directional SLOPE derivative $J'_{\boldsymbol{\lambda}}({\beta^0};u)$ is separable:
\begin{equation*}
    J'_{\boldsymbol{\lambda}}({\beta^0};u) = J^{I_0}_{\boldsymbol{\lambda}}(u) + J^{I_1}_{\boldsymbol{\lambda},{\beta^0}}(u) + ... + J^{I_m}_{\boldsymbol{\lambda},{\beta^0}}(u),
\end{equation*} 
with
\begin{align*}
     J^{I_0}_{\boldsymbol{\lambda}}(u) & =\sum_{i\in I_0}\boldsymbol{\lambda}_{\pi(i)}\vert u_i\vert, \\
      J^{I_j}_{\boldsymbol{\lambda},{\beta^0}}(u) & =\sum_{i\in I_j}\boldsymbol{\lambda}_{\pi(i)}u_i sgn(\beta^0_i),
\end{align*}
where the permutation $\pi$ in $J'_{\boldsymbol{\lambda}}({\beta^0};u)$ sorts the limiting pattern of $u$ w.r.t. $\beta^0$, i.e; $\vert \mathfrak{p}_0 \vert_{\pi^{-1}(1)}\geq\dots\geq\vert \mathfrak{p}_0 \vert_{\pi^{-1}(p)}$, with $\mathfrak{p}_0=\mathbf{patt}_{\beta^0}(u)$.

 Hence 
 \begin{equation*}
     \text{prox}_{J_{\boldsymbol{\lambda},\beta^0}}(y)=\text{prox}_{J^{I_0}_{\boldsymbol{\lambda},{\beta^0}}}(y)\oplus\text{prox}_{J^{I_1}_{\boldsymbol{\lambda},{\beta^0}}}(y)\oplus\dots\oplus\text{prox}_{J^{I_m}_{\boldsymbol{\lambda},{\beta^0}}}(y)
 \end{equation*}
 
 Since we can treat each cluster separately, we can w.l.o.g. assume that $\beta^0$ consists of one cluster only. There are only two possible cases:
 
 In the first case, $\beta^0=0$ and the proximal operator is described in \cite{bogdan2015slope}:
\begin{align}
    \text{prox}_{J_{\boldsymbol{\lambda}}}(y)&=\underset{u\in\mathbb{R}^p}{\operatorname{argmin }}\hspace{0.2cm} (1/2)\Vert u-y\Vert_2^2 + J_{\boldsymbol{\lambda}}(u)\nonumber\\
   &=S_{y}\Pi_{y}\underset{\tilde{u}_1\geq\dots\geq\tilde{u}_p\geq0 }{\operatorname{argmin }}\hspace{0.2cm} (1/2)\Vert \tilde{u}-\vert y\vert_{(\cdot)}\Vert_2^2 + \sum_{i=1}^p\lambda_i\tilde{u}_i,\label{slope proximal operator optimization}
\end{align}
where $\vert y\vert_{(\cdot)}=\Pi^T_yS_{y}y$  arises by sorting the absolute values of $y$. (See Proposition 2.2 in \cite{bogdan2015slope} and notation in the section Subdifferential and Pattern.)

In the second case, $\beta^0$ consists of a single non zero cluster. In this case the penalty becomes $J'_{\boldsymbol{\lambda}}({\beta^0};u)=\sum_{i=1}^p \boldsymbol{\lambda}_{\pi(i)}(S_{\beta^0}u)_i=\sum_{i=1}^p \boldsymbol{\lambda}_{i}(S_{\beta^0}u)_{\pi^{-1}(i)}$, where $ (S_{\beta^0}u) _{\pi^{-1}(1)}\geq\dots\geq (S_{\beta^0}u) _{\pi^{-1}(p)}$. In particular, $J_{\boldsymbol{\lambda},\beta^0}(S_{\beta^0}u)=\sum_{i=1}^p \boldsymbol{\lambda}_{i}u_{\pi^{-1}(i)}=\sum_{i=1}^p \boldsymbol{\lambda}_{i}u_{(i)}$, with $ u _{(1)}\geq\dots\geq u _{(p)}$.

\begin{align}
    \text{prox}_{J_{\boldsymbol{\lambda},\beta^0}}(y)&=\underset{u\in\mathbb{R}^p}{\operatorname{argmin }}\hspace{0.2cm} (1/2)\Vert u-y\Vert_2^2 + J'_{\boldsymbol{\lambda}}({\beta^0};u)\nonumber\\
   & = S_{\beta^0}\underset{\tilde{u}\in\mathbb{R}^p }{\operatorname{argmin }}\hspace{0.2cm} (1/2)\Vert \tilde{u}- S_{\beta^0}y\Vert_2^2 + \sum_{i=1}^p\lambda_i\tilde{u}_{(i)}\nonumber\\
   & =  S_{\beta^0}\Pi\underset{\tilde{u}_1\geq\dots\geq\tilde{u}_p}{\operatorname{argmin }}\hspace{0.2cm} (1/2)\Vert \tilde{u}-(S_{\beta^0}y)_{(\cdot)}\Vert_2^2 + \sum_{i=1}^p\lambda_i\tilde{u}_i,\label{isotonic optimization}
\end{align}

where $(S_{\beta^0}y)_{(\cdot)}=\Pi ^T (S_{\beta^0}y)$ is the sorted \footnote{Note that the permutation matrix $\Pi$ depends both on $y$ and $\beta^0$. } $(S_{\beta^0}y)$ vector. The optimization problem in (\ref{isotonic optimization}) is very similar to the optimization problem in (\ref{slope proximal operator optimization}). The only difference is in the relaxed constraint, where the set of feasible solutions in (\ref{isotonic optimization}) allows for negative values. The optimization (\ref{isotonic optimization}) is a special case of the isotonic regression problem \cite{10.2307/2284712}:
\begin{align*}
\operatorname{minimize }\hspace{0.5cm} & \Vert x-z\Vert_2^2\nonumber \\
\text{subject to}\hspace{0.5cm} & x_1\geq\dots\geq x_p,
\end{align*}
where we set $z= (S_{\beta^0}y)_{(\cdot)} - \boldsymbol{\lambda} $.

\end{appendix}

\end{document}